\documentclass[11pt]{article}

\usepackage{amsmath,amssymb,amsfonts,amsbsy,mathrsfs}

\renewcommand{\subsection}{\subsubsection}

\newtheorem{theorem}{Theorem}[section]

\newtheorem{proposition}{Proposition}[section]
\newenvironment{proof}{\noindent{\bf Proof}.}{\hfill $\Box$ \vspace*{4mm}}

\newtheorem{remark}{Remark}[section]

\def\div{{\rm div}\,}
\def\diag{{\rm diag}\,}
\def\a{\alpha}
\def\nt{|\hspace{-0.7pt}|\hspace{-0.7pt}|}

\begin{document}

\begin{center}
{\LARGE \bf Stability of relativistic plasma-vacuum interfaces}
\end{center}

\vspace*{7mm}

\centerline{\large Yuri Trakhinin}

\begin{center}
Sobolev Institute of Mathematics, Koptyug av. 4, 630090 Novosibirsk, Russia\\
and\\
Novosibirsk State University, Pirogova str. 2, 630090 Novosibirsk, Russia\\
e-mail: trakhin@math.nsc.ru
\end{center}

\vspace*{7mm}
\centerline{\bf \large Abstract}
\noindent We study the plasma-vacuum interface problem in relativistic magnetohydrodynamics for the case when the plasma density does not go to zero continuously, but jumps. In the vacuum region we consider the Maxwell equations for electric and magnetic fields. We show that a sufficiently large vacuum electric field can make the planar interface violently unstable. By using a suitable secondary symmetrization of the vacuum Maxwell equations, we find a sufficient condition that precludes violent instabilities. Under this condition we derive an energy a priori estimate in the anisotropic weighted Sobolev space $H^1_*$ for the variable coefficients linearized problem for nonplanar plasma-vacuum interfaces.

\section{Introduction}
\label{sec:1}

Consider the equations of relativistic magnetohydrodynamics (RMHD) for an ideal fluid (see, e.g., \cite{An,Lich}):
\begin{equation}
\nabla_{\a}(\rho u^{\a})=0, \qquad
\nabla_{\a}T^{\a\beta}=0,
\qquad
\nabla_{\a}(u^{\a}b^{\beta} -u^{\beta}b^{\a})=0,
\label{1}
\end{equation}
where $\nabla_{\a}$ is the covariant derivative with respect to
the Lorentzian metric $g=\diag (-1,1,1,1)$ of the space-time with the
components $g_{\a\beta}$,\footnote{We restrict ourselves to the case of {\it special relativity}, but in the future we hope to extend our results to general relativity (see discussion in Section \ref{sec:9}).} $\rho$ is the proper rest-mass density of the plasma, $u^{\a}$ are components of the 4-velocity,
\[
T^{\a\beta}=(\rho h +B^2)u^{\a}u^{\beta} +qg^{\a\beta} -b^{\a}b^{\beta},
\]
$h= 1+e +(p/\rho )$ is the relativistic specific enthalpy, $p$ is the pressure, $e=e(\rho ,S)$ is the specific internal energy, $S$ is the specific entropy, $B^2=b^{\a}b_{\a}$, $b^{\a}$ are components of the magnetic field 4-vector with respect to the plasma velocity, and $q= p+\frac{1}{2}B^2$ is the total pressure. The 4-vectors satisfy the conditions
$u^{\a}u_{\a} =-1$ and $u^{\a}b_{\a} =0$, and the speed of the light is equal to unity.

Let $(x^0,x)$ be inertial coordinates, with $t=x^0$ a time coordinate and $x=(x^1,x^2,x^3)$ space coordinates. Then
\[
u^0=-u_0=\Gamma ,\quad u^i=u_i =\Gamma v_i,\quad u:=(u_1,u_2,u_3)=\Gamma v,
\]
\[
\Gamma^2=1+|u|^2,\quad
b^0=-b_0=(u,H),\quad
b^i=b_i=\frac{H_i}{\Gamma}+(u,H)v_i,
\]
\[
b:=(b_1,b_2,b_3),\quad
B^2=|b|^2-b_0^2=\frac{|H|^2}{\Gamma^2} +(v,H)^2>0,
\]
where $\Gamma=(1-|v|^2)^{-1/2}$ is the Lorentz factor, $v=(v_1,v_2,v_3)$ is the plasma 3-velocity, $H=(H_1,H_2,H_3)$ is the magnetic field 3-vector, the Latin indices run from 1 to 3, and by $(\ ,\ )$ we denote the scalar product.

The RMHD equations \eqref{1} can be now rewritten in the form of the following system of conservation laws:
\begin{align}
& \partial_t(\rho\Gamma ) +\div (\rho u )=0,\label{2}\\
& \partial_t\left(\rho h \Gamma u +|H|^2v-(v,H)H\right) \nonumber \\
&  + {\rm div} \left((\rho h +B^2) u\otimes u  -b\otimes b\right) + {\nabla}q=0,\label{3}\\
& \partial_t(\rho h\Gamma^2 +|H|^2-q ) +{\rm div} \left(\rho h\Gamma u +|H|^2v-(v,H)H \right)=0,\label{4}\\
& \partial_tH -{\nabla}\times (v {\times} H)=0,\label{5}
\end{align}
where $\partial_t=\partial /\partial t$, $\nabla =(\partial_1,\partial_2,\partial_3)$, $\partial_i=\partial /\partial x^i$, etc.
Moreover, the first equation in the Maxwell equations containing in \eqref{1},
\begin{equation}
\div H=0,\label{6}
\end{equation}
should be now considered as the divergence constraint on the initial data $U (0,x)=U_0(x)$ for the unknown $U=(p,u,H,S)$.

Using \eqref{6} and the additional conservation law (entropy conservation)
\begin{equation}
\partial_t(\rho\Gamma S ) +\div (\rho S u )=0\label{7}
\end{equation}
which holds on smooth solutions of system \eqref{2}--\eqref{5}, and following {\sc Godunov}'s symmetrization procedure \cite{Go}, we can symmetrize the conservation laws \eqref{2}--\eqref{5} in terms of a vector of canonical variables $Q=Q(U)$. This was done by {\sc Ruggeri} \& {\sc Strumia} \cite{RuSt} and also by {\sc Anile} \& {\sc Pennisi} \cite{AP,An}. However, a concrete form of symmetric matrices was not found in \cite{RuSt,AP}. Moreover, if we deal with an initial-boundary value problem it is very inconvenient to work in terms of the vector $Q$.

In principle, from  $Q$ we can return to the vector of {\it primitive variables} $U$ keeping the symmetry property (see \cite{BT}). But, for RMHD finding a concrete form of symmetric matrices associated to the vector $Q$ and returning from $Q$ to $U$ is connected with unimaginable (or even almost unrealizable in practice) calculations. To overcome this serious technical difficulty {\sc Freist\"uhler} \& {\sc Trakhinin} \cite{FT} symmetrized recently the RMHD system by using the Lorentz transform. Namely, taking into account \eqref{6}, for the fluid rest frame ($v=0$) we can rewrite \eqref{2}, \eqref{3}, \eqref{5}, \eqref{7} in a nonconservative form which will be already a symmetric system. After that we should properly apply the Lorentz transform to this system to get a corresponding symmetric system in the LAB-frame. Referring to \cite{FT} for details, here we just write down the resulting symmetrization of RMHD:
\begin{equation}
\label{8}
A_0(U )\partial_tU+\sum_{j=1}^3A_j(U )\partial_jU=0,
\end{equation}
where
\[
A_0=\left(
\begin{array}{cccc}
{\displaystyle\frac{\Gamma}{\rho a^2}} & v^{\sf T} &0 &0 \\[3pt]
v & \mathcal{A} &0 &0 \\
0& 0 & \mathcal{M} &0 \\
0 &0&0 & 1
\end{array}\right),\qquad A_j=\left(
\begin{array}{cccc}
{\displaystyle\frac{u_j}{\rho a^2}} & e_j^{\sf T} &0 &0\\[3pt]
e_j & \mathcal{A}_j &{\mathcal{N}_j}^{\sf T} &0\\
0& \mathcal{N}_j & v_j\mathcal{M} &0 \\
0&0&0&v_j
\end{array}\right),
\]
$a^2=p_{\rho}(\rho ,S)$, the unit column vectors $e_j= (\delta_{1j},\delta_{2j},\delta_{3j})$,
\[
\mathcal{A}=\left( \rho h \Gamma +\frac{|H|^2}{\Gamma} \right) I - \left( \rho h \Gamma +\frac{|H|^2+B^2}{\Gamma}\right) v\otimes v   -\frac{1}{\Gamma}\, H\otimes H +
\frac{(v,H)}{\Gamma}
\bigl( v\otimes H + H\otimes v\bigr) ,
\]
\[
\mathcal{M} =\frac{1}{\Gamma}\, (I +u\otimes u),\qquad \mathcal{N}_j= \frac{1}{\Gamma}\,b\otimes e_j -\frac{v_j}{\Gamma}\, b\otimes v - \frac{H_j}{\Gamma^2}\,I,
\]
\begin{multline*}
\mathcal{A}_j=
v_j\left\{ \left( \rho h \Gamma +\frac{|H|^2}{\Gamma} \right) I - \left( \rho h \Gamma +\frac{|H|^2-B^2}{\Gamma}\right) v\otimes v  -\frac{1}{\Gamma}\, H\otimes H\right\}\\
+\frac{H_j}{\Gamma}\left\{ \frac{1}{\Gamma^2}
\bigl( v\otimes H + H\otimes v\bigr) -2(v,H)(I-v\otimes v) \right\}
\\
+\frac{(v,H)}{\Gamma}\left( H\otimes e_j + e_j \otimes H\right) -\frac{B^2}{\Gamma}
\left( v\otimes e_j + e_j \otimes v\right),
\end{multline*}
and $I$ is the unit matrix. Clearly, $A_{\alpha}$ are symmetric matrices. Note also that the last equation in \eqref{8} is just the nonconservative form ${\rm d}S/{\rm d}t=0$ of \eqref{7}, with ${\rm d}/{\rm d}t=\partial_t +(v,\nabla )$.

As is known \cite{An}, natural physical restrictions guaranteeing the hyperbo\-licity of the RMHD system do not depend on the magnetic field and coincide with corresponding ones in relativistic gas dynamics. In our case, by direct calculations one can show that the hyperbolicity condition $A_0>0$ holds provided that
\begin{equation}
\rho >0,\quad p_{\rho} >0,\quad 0<c_s^2<1\label{9}
\end{equation}
(of course, by default we also assume that $|v|<1$), where $c_s$ is the relativistic speed of sound, $c_s^2= a^2/h=p_{\rho}/h$. The last inequality in \eqref{9} is the {\it relativistic causality} condition.

Plasma-vacuum interface problems for the {\it nonrelativistic} MHD system usually appear in the mathematical modeling of plasma confinement by magnetic fields. Most of theoretical studies were devoted to finding stability criteria of equilibrium states, and the typical work in this direction is the classical paper \cite{BFKK}. According to our knowledge the first mathematical study of the well-posedness of the full (non-stationary) plasma-vacuum model is the recent work \cite{T10}. This model consists in the MHD system in the plasma domain, the elliptic system $\nabla \times \mathcal{H}=0$, $\div \mathcal{H}=0$ (so-called {\it pre-Maxwell dynamics}) in the vacuum region, and corresponding interface boundary conditions \cite{BFKK,T10}, where by $\mathcal{H}$ we denote the vacuum magnetic field.

The relativistic version of the classical plasma-vacuum interface problem \cite{BFKK,T10} is of great interest in connection with various astrophysical applications such as, for example, {\it neutron stars}.
Unlike the nonrelativistic case, we cannot now neglect the displacement current $\partial_tE$ and must consider the full Maxwell equations in vacuum:
\begin{equation}
\partial_t\mathcal{H}+\nabla\times E =0,\qquad \partial_tE-\nabla\times \mathcal{H} =0,\label{10}
\end{equation}
where $E$ is the electric field in the vacuum and we recall that the speed of the light is equal to unity. We do not include the equations
\begin{equation}
\div\mathcal{H}=0,\qquad \div E=0\label{11}
\end{equation}
into the main system \eqref{10} because they are just divergence constraints on the initial data $V(0,x)=V_0(x)$ for the ``vacuum'' unknown $V=(\mathcal{H},E)$. This is true for the Cauchy problem, but below we will also prove this for our plasma-vacuum interface problem.

Let us now discuss boundary conditions which we should pose on the interface (free boundary)  between plasma and vacuum. Assume that this interface is a hypersurface
\begin{equation}
\Sigma (t)=\{ x^1=\varphi (t,x')\},\quad\mbox{with}\quad x'=(x^2,x^3).\label{12}
\end{equation}
Our assumption that interface \eqref{12} has the form of a graph is not a strong restriction for special relativity. Regarding the case of general relativity which we plan to consider in the future, the interface should be a compact  codimension-1 surface. However, as for shock waves we can follow {\sc Majda}'s arguments \cite{Maj83b} for extending the results to a compact plasma-vacuum interface (see also a related discussion in \cite{T09.cpam} for gas dynamics). On the other hand, even for general relativity, the above assumption is not still a strong restriction because, as in \cite{FR}, we can consider the interface not only locally in time (in the sense of a future local-in-time existence theorem), but also locally in space (for compactly supported initial data).\footnote{Of course, locally we may suppose that the interface has the form of a graph.}

Let $\Omega^{\pm} (t)=\left\{x^1\gtrless \varphi (t,x')\right\}$ are space-time domains occupied by the plasma and the vacuum respectively. That is, in the domain
$\Omega^+(t)$ we consider system \eqref{2}--\eqref{5} (or \eqref{8}) governing the motion of an ideal relativistic plasma and in the domain $\Omega^-(t)$ we have the Maxwell equations \eqref{10} in vacuum. The interface $\Sigma (t)$ is a free boundary and moves with the velocity of plasma particles at the boundary:
\begin{equation}
\partial_t\varphi =v_N \quad \mbox{on}\ \Sigma (t),\label{13}
\end{equation}
where $v_N=(v,N)$ and $N=(1,-\partial_2\varphi ,-\partial_3\varphi )$.

The relativistic counterpart of the interface boundary condition $[q]= q|_{\Sigma}-\frac{1}{2}|\mathcal{H}|^2_{|\Sigma}=0$ in \cite{BFKK,T10} is the condition
\begin{equation}
q=\frac{|\mathcal{H}|^2 -|E|^2}{2} \quad \mbox{on}\ \Sigma (t).\label{14}
\end{equation}
Indeed, remembering that the plasma electric field $E^+=-v\times H$, the relativistic total pressure can be rewritten in the form
\[
q=p+\frac{|H|^2}{2\Gamma^2} +\frac{(v,H)^2}{2}= p+\frac{|{H}|^2 -|E^+|^2}{2}.
\]
That is, the value $(|{H}|^2 -|E^+|^2)/2$ can be considered as the electromagnetic contribution
(electromagnetic pressure) to the relativistic total pressure in plasma, and in the right-hand side of equation \eqref{14} we have the corresponding value in vacuum. Note also that the values $|{H}|^2 -|E^+|^2$  and $|\mathcal{H}|^2 -|E|^2$ are Lorentz invariants, in particular,
$q=p+\frac{1}{2}|H'|^2$, where $H'=H|_{v=0}$ is the magnetic field in the fluid rest frame (and $E^+_{|v=0}=0$).

As in the nonrelativistic case \cite{BFKK,T10}, we have also the conditions
\begin{equation}
H_N=0\quad \mbox{on}\ \Sigma (t),
\label{15.1}
\end{equation}
\begin{equation}
\mathcal{H}_N=0\quad \mbox{on}\ \Sigma (t),
\label{15.2}
\end{equation}
where $\mathcal{H}_N=(\mathcal{H},N)$ and ${H}_N=({H},N)$. These conditions express the fact that the interface is a perfectly conducting free boundary. At the same time, as for current-vortex sheets \cite{T09} they are not real boundary conditions and should be considered as restrictions on the initial data (see Section \ref{sec:3} for the proof).

At last, again as in the nonrelativistic case \cite{BFKK}, we have the boundary conditions
\[
\mathcal{H}\partial_t\varphi = N\times E
\]
on $\Sigma (t)$. However, the first of these boundary conditions is reduced to \eqref{15.2}. That is, we have finally the following two boundary conditions
\begin{equation}
E_2=\mathcal{H}_3\partial_t\varphi -E_1\partial_2\varphi,\quad
E_3=-\mathcal{H}_2\partial_t\varphi -E_1\partial_3\varphi\quad \mbox{on}\ \Sigma (t)
\label{16}
\end{equation}
for the vacuum electric and magnetic fields. It is worth noting that these boundary conditions are just jump conditions which can be deduced for the system $\partial_tH^{\pm} +\nabla \times E^{\pm} =0$ in $\Omega^{\pm}(t)$, with $E^+=-v\times H$, $E^-=E$, $H^+=H$, and $H^-=\mathcal{H}$, i.e., for system \eqref{5} and the first system in \eqref{10}. To get the final form \eqref{16} we should take into account \eqref{13} and \eqref{15.1}.

Thus, we obtain the free boundary vale problem for system \eqref{8} in $\Omega^+(t)$ and system \eqref{10} in $\Omega^-(t)$ with the boundary conditions \eqref{13}, \eqref{14}, \eqref{16} on $\Sigma (t)$ and the initial  data
\begin{equation}
\begin{array}{c}
{U} (0,{x})={U}_0({x}),\quad {x}\in \Omega^{+} (0),\quad
{V} (0,{x})={V}_0({x}),\quad {x}\in \Omega^{-} (0),\\[3pt]
\varphi (0,{x}')=\varphi_0({x}'),\quad {x}\in\mathbb{R}^2.
\end{array}
\label{17}
\end{equation}
Moreover, we will prove that \eqref{6}, \eqref{11}, \eqref{15.1} and \eqref{15.2} are restrictions on the initial data \eqref{17}, i.e., if they are satisfied at the first moment $t=0$, then they hold for all $t>0$. System \eqref{10} is always hyperbolic and we assume that the hyperbolicity condition \eqref{9} is satisfied up to the boundary of the domain $\Omega^+(t)$. That is, as in
the fluid-vacuum problem in \cite{Lind,T09.cpam}, we consider the case when the density does not go to zero continuously, but jumps. However, unlike the problem in \cite{Lind,T09.cpam}, we do not have the boundary condition $p|_{\Sigma}=0$ and, therefore, the condition $\rho |_{\Sigma}>0$ does not contradict, for example, the equation of state of a polytropic gas. We may suppose that the jump $[\rho ]=\rho |_{\Sigma}$ is very small that corresponds to a small but non-zero pressure on the interface.

Our final goal is to find (at least sufficient) conditions on the initial data \eqref{17} providing the local-in-time existence and uniqueness of a solution $(U,V,\varphi )$ of problem
\eqref{8}, \eqref{10}, \eqref{13}, \eqref{14}, \eqref{16}, \eqref{17} in Sobolev spaces. The first step towards the proof of such an existence theorem is a careful study of a corresponding linearized problem. The main result of the present paper is finding a {\it sufficient stability condition} which gives a (basic) energy a priori estimate. Roughly speaking, we show that if the unperturbed vacuum electric field is small enough and the unperturbed plasma and vacuum magnetic fields are nowhere parallel to each other on the interface, then the interface is stable. It is noteworthy that this condition could be {\it essentially specified} by a numerical analysis of the positive definiteness of some matrix of order 42 (see \eqref{122} in Section \ref{sec:6}).  The crucial role in deriving a basic a priori estimate for the linearized problem is played by a secondary symmetrization of the Maxwell equations \eqref{10} (see Section \ref{sec:2}). Using the terminology of \cite{T06}, we construct a {\it dissipative 1-symmetrizer} for the linearized problem.  This allows us to find the mentioned stability condition. Moreover, we show that the corresponding constant coefficients linearized problem for a planar interface can be violently unstable for a sufficiently large vacuum electric field.

The main difficulty in the study of the relativistic plasma-vacuum problem is the same as for compressible current-vortex sheets \cite{T05,T09} for the MHD system, i.e., we cannot test the Kreiss-Lopatinski condition \cite{Kreiss,Maj83b} analytically. On the other hand, since, as for current-vortex sheets, the number of dimensionless parameters for the constant coefficients linearized problem is big, a complete numerical test of the Kreiss-Lopatinski condition seems unrealizable in practice. We overcome this principal difficulty by using the energy method and constructing the mentioned dissipative 1-symmetrizer.

Another principal difficulties are again the same as for current-vortex sheets. That is,
the Kreiss-Lopatinski condition can be satisfied only in a weak sense and the plasma-vacuum interface is a characteristic free boundary. Therefore, we have a loss of derivatives in our a priori estimates and a natural loss of control on derivatives in the normal direction that cannot be compensated in RMHD (unlike the situation in gas dynamics \cite{CS,T09.cpam}). Therefore, in our nonlinear analysis (this work is now in progress) we have to use a suitable Nash-Moser-type iteration scheme (as, for example, in \cite{CS,T09,T09.cpam}) and the natural functional setting is provided by the anisotropic weighted Sobolev spaces $H^m_*$ (see \cite{YM,Sec} and Section \ref{sec:4} for their definition). In this paper, we prove our basic a priori estimate in the space $H^1_*$ for the ``plasma'' unknown $U$, and for the ``vacuum'' unknown $V$ we have a full regularity in the normal direction and can work in the usual Sobolev space $H^1$.

The rest of the paper is organized as follows. In Section \ref{sec:2} we write down the mentioned secondary symmetrization of the vacuum Maxwell equations. In Section \ref{sec:3} we reduce our free boundary problem to an initial-boundary value problem in a fixed domain and discuss properties of the reduced problem. In Section \ref{sec:4}, we obtain the linearized problem and formulate the main result for it (Theorem \ref{t1}). In Section \ref{sec:5} we properly reduce the linearized problem to that with homogeneous boundary conditions and homogeneous Maxwell equations. Section \ref{sec:6} is devoted to the derivation of the basic a priori estimate for the reduced linearized problem in the case of constant coefficients. In Section \ref{sec:7}, by considering particular cases of the unperturbed constant solution, we prove that the constant coefficients linearized problem can be ill-posed. In Section \ref{sec:8}, we extend the result of Section \ref{sec:6} to the case of variable coefficients. At last, Section \ref{sec:9} contains concluding remarks.

\section{Secondary symmetrization of the vacuum Maxwell equations}
\label{sec:2}

The Maxwell equations \eqref{10} form the symmetric hyperbolic system
\begin{equation}
\partial_tV +\sum_{j=1}^3B_j\partial_jV=0,
\label{18}
\end{equation}
with
\[
B_1=\left(\begin{array}{cccccc}
0 & 0 & 0& 0 & 0 & 0 \\
0 & 0 & 0& 0 & 0 & -1 \\
0 & 0 & 0& 0 & 1 & 0 \\
0 & 0 & 0& 0 & 0 & 0 \\
0 & 0 & 1& 0 & 0 & 0 \\
0 & -1 & 0& 0 & 0 & 0
\end{array} \right),\quad
B_2=\left(\begin{array}{cccccc}
0 & 0 & 0& 0 & 0 & 1 \\
0 & 0 & 0& 0 & 0 & 0 \\
0 & 0 & 0& -1 & 0 & 0 \\
0 & 0 & -1& 0 & 0 & 0 \\
0 & 0 & 0& 0 & 0 & 0 \\
1 & 0 & 0& 0 & 0 & 0
\end{array} \right),
\]
\[
B_3=\left(\begin{array}{cccccc}
0 & 0 & 0& 0 & -1 & 0 \\
0 & 0 & 0& 1 & 0 & 0 \\
0 & 0 & 0& 0 & 0 & 0 \\
0 & 1 & 0& 0 & 0 & 0 \\
-1 & 0 & 0& 0 & 0 & 0 \\
0 & 0 & 0& 0 & 0 & 0
\end{array} \right).
\]
At the same time, they imply the wave equation for $\mathcal{H}$ (and $E$) which can be, in turn, symmetrized, i.e., it can be rewritten as a symmetric hyperbolic system for derivatives of $\mathcal{H}$ (and $E$). There are different ways to symmetrize the wave equation, and symmetrizations may depend on arbitrary constants or functions (see, e.g., \cite{BT,T06}). However, for our present goals it is better to use a counterpart of one of these symmetrizations for the original system \eqref{18}. Since the original system was already symmetric, it is natural to call its new symmetric form {\it secondary symmetrization}.

Consider \eqref{18} in the whole space $\mathbb{R}^3$. Clearly,
\[
\frac{\rm d}{{\rm d}t}\int_{\mathbb{R}^3}|V|^2\,{\rm d}x=0,
\]
but, using constraints \eqref{11}, by simple manipulations from \eqref{10} we can deduce the following additional conserved integrals:
\[
\frac{\rm d}{{\rm d}t}\int_{\mathbb{R}^3}(\mathcal{H}_2{E}_3-\mathcal{H}_3{E}_2)\,{\rm d}x=0,
\]
\[
\frac{\rm d}{{\rm d}t}\int_{\mathbb{R}^3}(\mathcal{H}_3{E}_1-\mathcal{H}_1{E}_3)\,{\rm d}x=0,\quad
\frac{\rm d}{{\rm d}t}\int_{\mathbb{R}^3}(\mathcal{H}_1{E}_2-\mathcal{H}_2{E}_1)\,{\rm d}x=0.
\]
Hence, we have
\begin{equation}
\frac{\rm d}{{\rm d}t}\int_{\mathbb{R}^3}\bigl\{|V|^2
+\nu_1(\mathcal{H}_2{E}_3-\mathcal{H}_3{E}_2)
+\nu_2(\mathcal{H}_3{E}_1-\mathcal{H}_1{E}_3)
+\nu_3(\mathcal{H}_1{E}_2-\mathcal{H}_2{E}_1)\bigr\}\,{\rm d}x=0,\label{19}
\end{equation}
where $\nu_i$ are yet arbitrary constants, but below they will be arbitrary functions satisfying a certain hyperbolicity condition.

The energy identity \eqref{19} also reads
\[
\frac{\rm d}{{\rm d}t}\int_{\mathbb{R}^3}(\mathcal{B}_0V,V)\,{\rm d}x=0,
\]
where
\[
\mathcal{B}_0=\left(\begin{array}{cccccc}
1 & 0 & 0& 0 & \nu_3 & -\nu_2 \\
0 & 1 & 0& -\nu_3 & 0 & \nu_1 \\
0 & 0 & 1& \nu_2 & -\nu_1 & 0 \\
0 & -\nu_3 & \nu_2& 1 & 0 & 0 \\
\nu_3 & 0 & -\nu_1& 0 & 1 & 0 \\
-\nu_2 & \nu_1 & 0& 0 & 0 & 1
\end{array} \right).
\]
We now assume that $\nu_i$ in $\mathcal{B}_0$ are arbitrary functions $\nu_i(t,x)$. Then, taking into account the divergence constraints \eqref{11}, it follows from \eqref{18} that
\begin{equation}
\mathcal{B}_0\partial_tV +\sum_{j=1}^3\mathcal{B}_0B_j\partial_jV +R_1\div\mathcal{H}
+R_2\div E
=\mathcal{B}_0\partial_tV +\sum_{j=1}^3\mathcal{B}_j\partial_jV=0,
\label{20}
\end{equation}
where
\[
\mathcal{B}_1=
\left(\begin{array}{cccccc}
\nu_1 & \nu_2 & \nu_3& 0 & 0 & 0 \\
\nu_2 & -\nu_1 & 0& 0 & 0 & -1 \\
\nu_3 & 0 & -\nu_1& 0 & 1 & 0 \\
0 & 0 & 0& \nu_1 & \nu_2 & \nu_3 \\
0 & 0 & 1& \nu_2 & -\nu_1 & 0 \\
0 & -1 & 0& \nu_3 & 0 & -\nu_1
\end{array} \right),
\quad
\mathcal{B}_2=
\left(\begin{array}{cccccc}
-\nu_2 & \nu_1 & 0& 0 & 0 & 1 \\
\nu_1 & \nu_2 & \nu_3& 0 & 0 & 0 \\
0 & \nu_3 & -\nu_2& -1 & 0 & 0 \\
0 & 0 & -1& -\nu_2 & \nu_1 & 0 \\
0 & 0 & 0& \nu_1 & \nu_2 & \nu_3 \\
1 & 0 & 0& 0 & \nu_3 & -\nu_2
\end{array} \right),
\]
\[
\mathcal{B}_3=
\left(\begin{array}{cccccc}
-\nu_3 & 0 & \nu_1& 0 & -1 & 0 \\
0 & -\nu_3 & \nu_2& 1 & 0 & 0 \\
\nu_1 & \nu_2 & \nu_3& 0 & 0 & 0 \\
0 & 1 & 0& -\nu_3 & 0 & \nu_1 \\
-1 & 0 & 0& 0 & -\nu_3 & \nu_2 \\
0 & 0 & 0& \nu_1 & \nu_2 & \nu_3
\end{array} \right),
\quad
R_1=\left(\begin{array}{c} \nu_1 \\
\nu_2 \\
\nu_3 \\
0 \\
0 \\
0
\end{array} \right),\quad R_2=\left(\begin{array}{c}
0 \\
0 \\
0 \\
\nu_1 \\
\nu_2 \\
\nu_3
\end{array} \right).
\]
The symmetric system \eqref{20} is hyperbolic if $\mathcal{B}_0>0$, i.e.,
\begin{equation}
|\nu |<1,\label{21}
\end{equation}
with the vector-function $\nu = (\nu_1, \nu_2, \nu_3)$. Following \cite{T09}, we can call the set $\mathbb{S}=(\mathcal{B}_0,R_1,R_2)$ the {\it secondary generalized Friedrichs symmetrizer}.

We will use \eqref{20} for obtaining a dissipative energy integral for the linearized plasma-vacuum problem, and for further arguments we will need the representations
\begin{equation}
\mathcal{B}_j=\mathcal{B}_0B_j +\mathcal{K}_j\label{22}
\end{equation}
following from the fact that
\[
R_1\div \mathcal{H}+R_2\div E=\sum_{j=1}^3\mathcal{K}_j\partial_jV,\qquad  \mathcal{K}_j=I_2\otimes (\nu \otimes e_j),
\]
where $I_2$ is the unit matrix of order 2.

\section{Reduced problem in a fixed domain and its properties}
\label{sec:3}

We now return to our free boundary value problem \eqref{8}, \eqref{13}, \eqref{14}, \eqref{16}--\eqref{18} and reduce it to an initial-boundary value problem in the fixed space
domain $\mathbb{R}^3_+=\{x^1>0,\ x'\in \mathbb{R}^2\}$. In \cite{T09,T09.cpam,T10}, to avoid assumptions about compact support of the initial data in the existence theorem and work glo\-bally in  $\mathbb{R}^3_+$ we used a cut-off function
$\chi (x^1)\in C_0^{\infty}( \mathbb{R})$ in the change of independent variables. Since in the future we plan to extend our results to the {\it general relativistic} version of the plasma-vacuum interface problem, here we prefer to apply a simplest straightening of the interface $\Sigma$ (without using the cut-off $\chi$). That is, in our future local-in-time existence theorem we will consider compactly supported initial data as, for example, in \cite{CS,FR}.

We make the simple change of variables $\tilde{x}^1=x^1-\varphi (t,x')$ for the RMHD system \eqref{8} and $\tilde{x}^1=-x^1+\varphi (t,x')$ for the Maxwell equations \eqref{10}. In other words, the unknowns $U$ and $V$ being smooth in $\Omega^{\pm}(t)$ are replaced by the vector-functions
\[
\widetilde{U}(t,x ):= {U}(t,x^1+\varphi (t,x'),x'),\quad
\widetilde{V}(t,x ):= V(t,-x^1+\varphi (t,x'),x'),
\]
which are smooth in the half-space $\mathbb{R}^3_+$. Dropping for convenience tildes, we reduce \eqref{8}, \eqref{13}, \eqref{14}, \eqref{16}--\eqref{18}  to the initial boundary value problem
\begin{equation}
\mathbb{P}(U,\varphi )=0\quad\mbox{in}\ [0,T]\times \mathbb{R}^3_+,\label{24}
\end{equation}
\begin{equation}
\mathbb{V}(V,\varphi )=0\quad\mbox{in}\ [0,T]\times \mathbb{R}^3_+,\label{25}
\end{equation}
\begin{equation}
\mathbb{B}(U,V,\varphi )=0\quad\mbox{on}\ [0,T]\times\{x^1=0\}\times\mathbb{R}^{2},\label{26}
\end{equation}
\begin{equation}
(U,V)|_{t=0}=(U_0,V_0)\quad\mbox{in}\ \mathbb{R}^3_+,\qquad \varphi|_{t=0}=\varphi_0\quad \mbox{in}\ \mathbb{R}^{2},\label{27}
\end{equation}
where $\mathbb{P}(U,\varphi )=L(U,\varphi )U$, $\mathbb{V}(V,\varphi )= {M}(\varphi )V$,
\[
L(U, \varphi )=A_0(U)\partial_t +\widetilde{A}_1(U,\varphi )\partial_1+A_2(U )\partial_2+A_3(U )\partial_3,
\]
\[
\widetilde{A}_1(U,\varphi )= A_1(U )-A_0(U)\partial_t\varphi -A_2(U)\partial_2\varphi
-A_3(U)\partial_3\varphi ,
\]
\[
{M}(\varphi )=I\,\partial_t -\widetilde{B}_1(\varphi )\partial_1+B_2\partial_2+B_3\partial_3,
\quad
\widetilde{B}_1(\varphi )= B_1-I\,\partial_t\varphi -B_2\partial_2\varphi
-B_3\partial_3\varphi ,
\]
\[\mathbb{B}(U,V,\varphi )=\left(
\begin{array}{c}
v_N -\partial_t\varphi  \\[3pt] q -\frac{1}{2}(|\mathcal{H}|^2-|E|^2) \\[3pt]
E_{\tau_2}-\mathcal{H}_3\partial_t\varphi  \\[3pt]
E_{\tau_3}+\mathcal{H}_2\partial_t\varphi
\end{array}
\right),\quad E_{\tau_i}=E_1\partial_i\varphi+E_i,\quad i=2,3.
\]

Regarding condition \eqref{15.1},
\begin{equation}
H_N|_{x^1=0}=0,\label{15.1'}
\end{equation}
 and the divergence constraint \eqref{6}, which in the straightened variables takes the form
\begin{equation}
\div {\sf h}=0,\label{28}
\end{equation}
with ${\sf h}=(H_N,H_2,H_3)$, referring to \cite{T09} for the proof, we have the following proposition.

\begin{proposition}
Let the initial data \eqref{27} satisfy \eqref{15.1'} and \eqref{28}.
If problem \eqref{24}--\eqref{27} has a solution $(U,V,\varphi )$, then this solution satisfies \eqref{15.1'} and \eqref{28} for all $t\in [0,T]$.
\label{p1}
\end{proposition}

In the proof of an analogous fact (see Proposition \eqref {p2} below) for the condition
\begin{equation}
\mathcal{H}_N|_{x^1=0}=0\label{15.2'}
\end{equation}
and equations \eqref{11} written in the straightened variables,
\begin{equation}
{\rm div}^- \,\mathfrak{h}=0,\quad {\rm div}^-\, \mathfrak{e} =0,\label{29}
\end{equation}
where $\mathfrak{h}=(\mathcal{H}_N,\mathcal{H}_2,\mathcal{H}_3)$, $\mathfrak{e}=({E}_N,E_2,E_3)$, $E_N=(E,N)$, and ${\rm div}^-\,a=-\partial_1a_1+\partial_2a_2+\partial_3a_3$ for any vector $a=(a_1,a_2,a_3)$, we should distinguish between two cases. In the first case
\begin{equation}
\partial_t\varphi \leq 0\label{expan1}
\end{equation}
and the strict inequality in \eqref{expan1} corresponds to the plasma expansion. In the opposite case
\begin{equation}
\partial_t\varphi > 0\label{shrink}
\end{equation}
we have shrinkage of the vacuum region. For the most general mixed case when for some parts of the interface the plasma expands into the vacuum and for other parts the vacuum expands into the plasma we below show that the boundary $x^1=0$ for the Maxwell system \eqref{25} is of {\it variable multiplicity}. We do not consider this very difficult case. At the same time, since, as already noted in Sect. 1, in our future local-in-time existence theorem we will consider the interface locally in space (for compactly supported initial data), it is reasonable to assume that one of the conditions in \eqref{expan1} and \eqref{shrink} is satisfied for the whole interface.

We now show that the boundary matrix $-\widetilde{B}_1(\varphi )$ for system \eqref{25} has different signatures for cases \eqref{expan1} and \eqref{shrink}. We can directly calculate its eigenvalues:
\[
\lambda_{1,2}= \partial_t\varphi +
\sqrt{1 +(\partial_2\varphi)^2+(\partial_3\varphi)^2},
\quad
\lambda_{3,4}= \partial_t\varphi-
\sqrt{1 +(\partial_2\varphi)^2+(\partial_3\varphi)^2},\qquad
\lambda_{5,6} =\partial_t\varphi .
\]
Since the interface speed should be less than the speed of the light, i.e.,
\[
\frac{|\partial_t\varphi |}{\sqrt{1 +(\partial_2\varphi)^2+(\partial_3\varphi)^2}} <1,
\]
the eigenvalues $\lambda_{1,2} >0$ and $\lambda_{3,4} <0$ for case \eqref{expan1} and $\lambda_{1,2} <0$ and $\lambda_{3,4} >0$ for case \eqref{shrink}. Note that we even have the more restrictive inequality
\begin{equation}
|\partial_t\varphi |<1
\label{36}
\end{equation}
because the first boundary condition in \eqref{26} for the particular case $v_2|_{x^1=0}=v_3|_{x^1=0}=0$ reads $\partial_t\varphi =v_1|_{x^1=0}$.

Let us count the number of incoming characteristics for systems \eqref{24} and \eqref{25}.
The simplest way to do this for the RMHD system \eqref{24} is to write down the quadratic form with the boundary matrix $\widetilde{A}_1({U},\varphi )$ on the boundary $x^1=0$ computed on a given basic state $(\widehat{U},\hat{\varphi})$ (``unperturbed flow'') about which we will linearize system \eqref{24}. Moreover, later on we will also need this quadratic form for deriving a priori estimates for the linearized problem by the energy method. We assume that the basic state satisfies \eqref{15.1'} and the first boundary condition  in \eqref{26}. In fact, because of this assumption the boundary $x^1=0$ is {\it characteristic} for system \eqref{24} and we now prove this.

It is convenient to use the following representations for the symmetric matrices $A_j$ in \eqref{8}:
\[
A_j = v_jA_0 + G_j,
\]
with
\[
G_j=\left(
\begin{array}{ccc}
0 & e_j^{\sf T}-v_jv^{\sf T} &0 \\[6pt]
e_j-v_jv & \mathcal{G}_j &{\mathcal{N}_j}^{\sf T} \\
0& \mathcal{N}_j & 0
\end{array}\right)
\]
and
\begin{multline*}
\mathcal{G}_j= v_j\left\{ 2\frac{B^2}{\Gamma}\,v\otimes v -\frac{(v,H)}{\Gamma}
\bigl( v\otimes H + H\otimes v\bigr)\right\}\\
+\frac{H_j}{\Gamma}\left\{ \frac{1}{\Gamma^2}
\bigl( v\otimes H + H\otimes v\bigr) -2(v,H)(I-v\otimes v) \right\}
\\
+\frac{(v,H)}{\Gamma}\left( H\otimes e_j + e_j \otimes H\right) -\frac{B^2}{\Gamma}
\left( v\otimes e_j + e_j \otimes v\right).
\end{multline*}
Then, taking into account the first boundary condition  in \eqref{26}, we have
\[
\mathcal{A}:=\widetilde{A}_1(\widehat{U},\hat{\varphi})|_{x^1=0}={G}_1(\widehat{U})|_{x^1=0}-{G}_2(\widehat{U})|_{x^1=0}\partial_2\hat{\varphi}-{G}_3(\widehat{U})|_{x^1=0}\partial_3\hat{\varphi}.
\]
Using \eqref{15.1'}, after some algebra we obtain
\begin{equation}
(\mathcal{A}U,U)=2\widehat{\Gamma}qv_N|_{x^1=0},
\label{35}
\end{equation}
where
\[
v_N=(v,\widehat{N}),\quad \widehat{N}=(1,-\partial_2\hat{\varphi},-\partial_3\hat{\varphi}),\quad
q= p +\frac{1}{2}\,\delta B^2,
\]
$v$ is the perturbation of $u/\Gamma$, and $\delta B^2$ is the perturbation of the electromagnetic pressure $B^2$:
\begin{equation}
v=\frac{1}{\widehat{\Gamma}}\left(u-(\hat{v},u)\hat{v}\right),\quad \delta B^2=\frac{2}{\widehat{\Gamma}}
\left( (\hat{b},H) +(\hat{v},\widehat{H})(\widehat{H},u)-\widehat{B}^2(\hat{v},u)\right).
\label{v,B}
\end{equation}
Here all the ``hat'' values are calculated on the basic state, e.g.,
$\hat{v}=\hat{u}/\widehat{\Gamma}$, $\hat{b}=\widehat{H}/\widehat{\Gamma}+(\hat{u},\widehat{H})\hat{v}$, etc.

Let $W=(q,\widehat{\Gamma}v_N,u_2,u_3,H,S)$ and $U= JW$. Clearly, $\det J\neq 0$. It follows from \eqref{35} that
\begin{equation}
(\mathcal{A}U,U)= (J^{\sf T}\mathcal{A}J\,W,W)=
(\mathcal{E}_{12}W,W),\qquad \mathcal{E}_{12}=\left(\begin{array}{ccccc}
0& 1 &0 & \cdots & 0 \\
1 & 0 &0 & \cdots & 0 \\
0 & 0 &0 & \cdots & 0 \\
\vdots & \vdots &\vdots & & \vdots \\
0& 0 &0 & \cdots & 0  \end{array}
 \right).
\label{35'}
\end{equation}
Hence, the matrix $\mathcal{A}$ has one positive (``incoming'') and one negative eigenvalue, and other eigenvalues are zeros.

In case \eqref{expan1} system \eqref{25} has two positive eigenvalues, and the whole system \eqref{24}, \eqref{25} for $U$ and $V$ has three incoming characteristics.  In this case the number of boundary conditions in \eqref{26} is correct because one of them is needed for determining the function $\varphi (t,x')$. If, however, the vacuum region expands, i.e., condition \eqref{shrink} holds, than the correct number of boundary conditions is six, but we have only four in \eqref{26}. It means that problem \eqref{24}--\eqref{27} is formally {\it underdetermined} and we miss two boundary conditions. However, if as these two missing boundary conditions we take
\begin{equation}
{\rm div}^-\, \mathfrak{h}|_{x^1=0} =0\qquad\mbox{and}\qquad
{\rm div}^-\, \mathfrak{e}|_{x^1=0} =0,
\label{add.bound}
\end{equation}
then, as for case \eqref{expan1}, we can show that equations \eqref{29} are just divergence constraints on the initial data \eqref{27}. Namely, the following proposition takes place.

\begin{proposition}
Let the initial data \eqref{27} satisfy \eqref{15.2'} and \eqref{29}. If problem \eqref{24}--\eqref{27} has a solution $(U,V,\varphi )$ with property \eqref{expan1}, then this solution satisfies \eqref{15.2'} and \eqref{29} for all $t\in [0,T]$. If problem \eqref{24}--\eqref{27} with the two additional boundary conditions \eqref{add.bound} has a solution $(U,V,\varphi )$ with property \eqref{shrink}, then this solution again satisfies \eqref{15.2'} and \eqref{29} for all $t\in [0,T]$.
\label{p2}
\end{proposition}

\begin{proof}
System \eqref{25} can be rewritten as follows
\begin{align}
& \partial_t\mathfrak{h}+\nabla^-\times \mathfrak{E}+\partial_t\varphi\,\partial_1\mathfrak{h} + B(\varphi )\mathfrak{h}=0,\label{30}  \\
& \partial_t\mathfrak{e}-\nabla^-\times \mathfrak{H}+\partial_t\varphi\,\partial_1\mathfrak{e}+ B(\varphi )\mathfrak{e}=0 , \label{31}
\end{align}
where
\[
\mathfrak{H}=(\mathcal{H}_1,\mathcal{H}_{\tau_2},\mathcal{H}_{\tau_3}),\quad
\mathfrak{E}=(E_1,E_{\tau_2},E_{\tau_3}),
\]
\[
{\mathcal{H}}_{\tau_i}={\mathcal{H}}_1\partial_i\varphi +{\mathcal{H}}_i,\quad i=2,3,\qquad
B=B(\varphi )=\left(
\begin{array}{ccc}
0 & \partial_t\partial_2\varphi & \partial_t\partial_3\varphi \\[2pt]
0 & 0 & 0\\
0 & 0 & 0
\end{array}
\right),
\]
and the operator $\nabla^-\times$ differs from the usual curl operator only by the fact that $\nabla^- =(-\partial_1,\partial_2,$ $\partial_3)$. Applying the operator ${\rm div}^-$ to \eqref{30} and \eqref{31}, we get the equation
\[
\partial_tF +\partial_t\varphi\, \partial_1F =0\quad\mbox{in}\ [0,T]\times \mathbb{R}^3_+
\]
for $F={\rm div}^-\,\mathfrak{h}\;$ and $F={\rm div}^-\,\mathfrak{e}$. In case \eqref{expan1} this equation has no incoming characteristics and, therefore, it does not need boundary conditions. Hence, if $F|_{t=0}=0$, then $F=0$ for all $t\in [0,T]$. This yields \eqref{29}. In case \eqref{shrink}, in view of the boundary conditions \eqref{add.bound}, we have $F|_{x^1=0}=0$. Then, we again deduce \eqref{29} if $F|_{t=0}=0$.

The first equation in \eqref{30} reads
\[
\partial_t\mathcal{H}_N+\partial_t\varphi\,\partial_1\mathcal{H}_N+\partial_2E_{\tau_3}-
\partial_3E_{\tau_2}+\mathcal{H}_2\partial_t\partial_2\varphi +\mathcal{H}_3\partial_t\partial_3\varphi =0.
\]
Considering it on the boundary $x^1=0$ and using the last two boundary conditions in \eqref{26}, we obtain
\[
\partial_t\mathcal{H}_N -\partial_t\varphi\,{\rm div}^- \,\mathfrak{h} =0\quad\mbox{at}\ x^1=0.
\]
In view of \eqref{29}, this implies $\partial_t\mathcal{H}_N|_{x^1=0}=0$. That is, we get \eqref{15.2'} for all $t\in [0,T]$ if $\mathcal{H}_N|_{x^1=0}=0$ for $t=0$.
\end{proof}

Below without loss of generality we consider the case \eqref{expan1}. The proof of the basic a priori estimate for the linearized problem in Sect. \ref{sec:6} and \ref{sec:8} for the opposite case \eqref{shrink} stays the same as for \eqref{expan1}. However, since the additional boundary conditions \eqref{add.bound} are not quite standard, the preparatory arguments in Sect. \ref{sec:5} should be modified and the future proof of the existence of solutions will be diferent from that for case \eqref{expan1}. Anyway, in this paper we restrict ourselves to case \eqref{expan1}. Moreover, it is natural to assume that the ``stable'' counterpart
\begin{equation}
\partial_t\varphi < 0\label{expan}
\end{equation}
of condition \eqref{expan1} holds. In particular, in our future local-in-time existence theorem we will consider initial data satisfying condition \eqref{expan}. We note also that under assumption \eqref{expan} the boundary $x^1=0$ is noncharacteristic for system \eqref{25}.\footnote{Even if $\partial_t\varphi = 0$, the boundary is characteristic for \eqref{25} only formally because,
in view of \eqref{29}, we have the full control on derivatives of $V$ in the normal direction.}

\section{Linearized problem
and main result}
\label{sec:4}

Let
\begin{equation}
(\widehat{U}(t,x ),\widehat{V}(t,x ),\hat{\varphi}(t,{x}'))
\label{37}
\end{equation}
be a given sufficiently smooth vector-function with $\widehat{U}=(\hat{p},\hat{u},\widehat{H},\widehat{S})$, $\widehat{V}=(\widehat{\mathcal{H}},\widehat{E})$, and
\begin{equation}
\|(\widehat{U},\widehat{V})\|_{W^2_{\infty}(\Omega_T)}+
\|\partial_1(\widehat{U},\widehat{V})\|_{W^2_{\infty}(\Omega_T)}+
\|\hat{\varphi}\|_{W^3_{\infty}(\partial\Omega_T)} \leq K,
\label{38}
\end{equation}
where $K>0$ is a constant,
\[
\Omega_T:= (-\infty, T]\times\mathbb{R}^3_+,\quad \partial\Omega_T:=(-\infty ,T]\times\{x^1=0\}\times\mathbb{R}^{2},
\]
and below all the ``hat'' values will be related to the basic state \eqref{36}, e.g.,
$\hat{v}=\hat{u}/\widehat{\Gamma}$, $\widehat{\Gamma}^2=1+|\hat{u}|^2$, $\hat{a}^2=1/\rho_p(\hat{p},\widehat{S})$, etc.

We assume that the basic state \eqref{37} satisfies the hyperbolicity condition \eqref{9} in $\overline{\Omega_T}$,
\begin{equation}
\rho (\hat{p},\widehat{S}) >0,\quad \rho_p(\hat{p},\widehat{S}) >0 ,\quad 0 <\hat{c}_s^2<1,
\label{39}
\end{equation}
together with such natural assumptions as $|\hat{v}|<1$, inequality \eqref{36} for $\hat{\varphi}$, etc.
We also assume that \eqref{37} satisfies the boundary conditions in \eqref{26} except the second one on $\partial\Omega_T$,
\begin{equation}
\hat{v}_N|_{x^1=0} =\varkappa ,\quad \widehat{E}_{\tau_2}|_{x^1=0}=
\varkappa\widehat{\mathcal{H}}_3|_{x^1=0} , \quad \widehat{E}_{\tau_3}|_{x^1=0}=-
\varkappa\widehat{\mathcal{H}}_2|_{x^1=0},
\label{40}
\end{equation}
and the interior equations for $\widehat{H}$ and $\widehat{\mathcal{H}}$ in ${\Omega_T}$ containing in \eqref{24} and \eqref{25}:
\begin{equation}
\partial_t\widehat{H}+ (\hat{w} ,\nabla )
\widehat{H} - (\hat{\sf h} ,\nabla ) \hat{v} + \widehat{H}\div\hat{w} =0,
\label{41}
\end{equation}
\begin{equation}
\partial_t\hat{\mathfrak{h}}+\varkappa \,\partial_1\hat{\mathfrak{h}}+\nabla^-\times \widehat{\mathfrak{E}} + B(\hat{\varphi} )\hat{\mathfrak{h}}=0,
\label{42}
\end{equation}
where $\varkappa =\partial_t\hat{\varphi},\quad \hat{w}=(\hat{v}_N- \varkappa ,\hat{v}_2,\hat{v}_3)$. At last, we assume that the basic state satisfies condition \eqref{expan}. More precisely, let
\begin{equation}
\varkappa \leq -\delta <0
\label{expan.lin}
\end{equation}
on $\partial\Omega_T$ with a fixed constant $\delta$. Recall also that $|\varkappa |<1$, cf. \eqref{36}.

Later on, for the linearized problem we will need equations associated to the nonlinear constraints \eqref{15.1'}, \eqref{28}, \eqref{15.2'}, and \eqref{29}. However, to deduce them it is not enough that these constraints are satisfied by the basic state \eqref{37}. For the plasma magnetic field, as in \cite{T09}, we need \eqref{41}, and for the same reason for the vacuum magnetic field we have to use \eqref{42}.
If we assume that the basic state satisfies
\begin{equation}
\widehat{H}_N|_{x^1=0}=0,\quad \widehat{\mathcal{H}}_N|_{x^1=0}=0,\label{43}
\end{equation}
\begin{equation}
\div \hat{\sf h}=0,\quad {\rm div}^- \,\hat{\mathfrak{h}}=0\label{44}
\end{equation}
for $t=0$, then it follows from the proofs of Propositions \ref{p1} (see \cite{T09}) and \ref{p2} that conditions \eqref{43} are true on $\partial\Omega_T$ and equations \eqref{44} hold in ${\Omega_T}$. Thus, without loss of generality for the basic state we also require the fulfillment of \eqref{43} on $\partial\Omega_T$ and \eqref{44} in $\Omega_T$.

Conditions \eqref{43} and the first condition in \eqref{40} endow the boundary matrix
\[
\diag (\widetilde{A}_1(\widehat{U},\hat{\varphi}),-\widetilde{B}_1(\hat{\varphi}))|_{x^1=0}
\]
for the linearized problem with the same structure as for the boundary matrix at $x^1=0$ for the nonlinear problem \eqref{24}--\eqref{27}. Concerning the second and third conditions in \eqref{40}, we will need them for deducing a linearized version of \eqref{15.2'} for solutions of the linearized problem.

\begin{remark}{\rm
Assumptions \eqref{39}--\eqref{expan.lin} are nonlinear constraints on the basic state. In the forthcoming nonlinear analysis we plan to use the Nash-Moser method. As in \cite{CS,T09,T09.cpam}, the Nash-Moser procedure will be not completely standard. Namely, at each $n$th Nash-Moser iteration step we will have to construct an intermediate state $(U_{n+1/2},V_{n+1/2},\varphi_{n+1/2})$ satisfying constraints  \eqref{39}--\eqref{expan.lin} (and \eqref{43}, \eqref{44} for $t=0$). For the ``plasma'' part of the problem such an intermediate state can be constructed in exactly the same manner as in \cite{T09} and for the ``vacuum'' component $V_{n+1/2}$ the construction is also not difficult. We omit corresponding arguments and postpone them to the nonlinear analysis.}
\label{r1}
\end{remark}

The linearized equations for \eqref{24}--\eqref{26} read:
\[
\mathbb{P}'(\widehat{U},\hat{\varphi})(\delta U,\delta\varphi ):=
\frac{\rm d}{{\rm d}\varepsilon}\mathbb{P}(U_{\varepsilon},\varphi_{\varepsilon})|_{\varepsilon =0}=f_{\rm I}
\quad \mbox{in}\ \Omega_T,
\]
\[
\mathbb{V}'(\widehat{V},\hat{\varphi})(\delta V,\delta\varphi ):=
\frac{\rm d}{{\rm d}\varepsilon}\mathbb{V}(V_{\varepsilon},\varphi_{\varepsilon})|_{\varepsilon =0}=f_{\rm II}
\quad \mbox{in}\ \Omega_T,
\]
\[
\mathbb{B}'(\widehat{U},\widehat{V},\hat{\varphi})(\delta U,\delta V,\delta \varphi ):=
\frac{\rm d}{{\rm d}\varepsilon}\mathbb{B}(U_{\varepsilon},V_{\varepsilon},\varphi_{\varepsilon})|_{\varepsilon =0}={g}
\quad \mbox{on}\ \partial\Omega_T,
\]
where $U_{\varepsilon}=\widehat{U}+ \varepsilon\,\delta U$, $V_{\varepsilon}=
\widehat{V}+\varepsilon\,\delta V$, and $\varphi_{\varepsilon}=\hat{\varphi}+ \varepsilon\,\delta \varphi$. Here we introduce the source terms $f=(f_{\rm I}, f_{\rm II})=(f_1,\ldots ,f_{14})$ and $g=(g_1,\ldots ,g_4)$ to make the interior equations and the boundary conditions inhomogeneous. We easily compute the exact form of the linearized equations (below we drop $\delta$):
\[
\mathbb{P}'(\widehat{U},\hat{\varphi})(U,\varphi )
=
L(\widehat{U},\hat{\varphi})U +{\cal C}(\widehat{U},\hat{\varphi})
U -   \bigl\{L(\widehat{U},\hat{\varphi})\varphi\bigr\}\partial_1\widehat{U}=f_{\rm I},
\]
\[
\mathbb{V}'(\widehat{V},\hat{\varphi})(V,\varphi )
=
M(\hat{\varphi})V  +   \bigl\{M(\hat{\varphi})\varphi\bigr\}\partial_1\widehat{V}=f_{\rm II},
\]
\[
\mathbb{B}'(\widehat{U},\widehat{V},\hat{\varphi})(U,V,\varphi )=
{\left(
\begin{array}{c}
v_{N}-\partial_t\varphi -\hat{v}_2\partial_2\varphi-\hat{v}_3\partial_3\varphi \\[3pt]
q-(\widehat{\mathcal{H}},\mathcal{H})+(\widehat{E},E)\\[3pt]
E_{\tau_2} -\widehat{\mathcal{H}}_3\partial_t\varphi -\varkappa\mathcal{H}_3+\widehat{E}_1\partial_2\varphi\\[3pt]
E_{\tau_3}+\widehat{\mathcal{H}}_2\partial_t\varphi +\varkappa\mathcal{H}_2+\widehat{E}_1\partial_3\varphi
\end{array}
\right)_{\bigl|}}_{x^1=0}=g,
\]
where
\[
q=p+ \delta B^2,\quad  v_{N}= v_1-v_2\partial_2\hat{\varphi}-v_3\partial_3\hat{\varphi},
\quad \quad E_{\tau_j}=E_1\partial_j\hat{\varphi}+E_j,
\]
$v$ and $\delta B^2$ are written down in \eqref{v,B},
and the matrix ${\cal C}(\widehat{U},\hat{\varphi})$ is determined as follows:
\begin{multline*}
{\cal C}(\widehat{U},\hat{\varphi})Y
= (Y ,\nabla_yA_0(\widehat{U} ))\partial_t\widehat{U}
 +(Y ,\nabla_y\widetilde{A}_1(\widehat{U},\hat{\varphi}))\partial_1\widehat{U}
 \\
+ (Y ,\nabla_yA_2(\widehat{U} ))\partial_2\widehat{U}
+ (Y ,\nabla_yA_3(\widehat{U} ))\partial_3\widehat{U},
\end{multline*}
\[
(Y ,\nabla_y A(\widehat{U})):=\sum_{i=1}^8y_i\left.\left(\frac{\partial A (Y )}{
\partial y_i}\right|_{Y =\widehat{U}}\right),\quad Y =(y_1,\ldots ,y_8).
\]

The differential operators $\mathbb{P}'(\widehat{U},\hat{\varphi})$ and $\mathbb{V}'(\widehat{V},\widehat{\varphi})$ are first-order operators in $\varphi$ that gives some trouble in obtaining  a  priori estimates. As in \cite{CS,T05,T09,T09.cpam}, following Alinhac \cite{Al}, we introduce the ``good unknowns''
\begin{equation}
\dot{U}:=U -\varphi\,\partial_1\widehat{U},\qquad \dot{V}:=V +\varphi\,\partial_1\widehat{V}.
\label{45}
\end{equation}
Omitting detailed calculations, we rewrite our linearized interior equations in terms of the new unknowns \eqref{45}:
\begin{equation}
L(\widehat{U},\hat{\varphi})\dot{U} +{\cal C}(\widehat{U},\hat{\varphi})
\dot{U} + \varphi\,\partial_1\bigl\{\mathbb{P}
(\widehat{U},\hat{\varphi})\bigr\}=f_{\rm I},
\label{46}
\end{equation}
\begin{equation}
M(\hat{\varphi})\dot{V}  - \varphi\,\partial_1\bigl\{\mathbb{V}
(\widehat{V},\hat{\varphi})\bigr\}=f_{\rm II},
\label{47}
\end{equation}
As in \cite{Al,CS,T09,T09.cpam}, we now drop the zero-order terms in $\varphi$ in \eqref{46} and \eqref{47}\footnote{In the future nonlinear analysis the dropped terms in \eqref{46} and \eqref{47} should be considered as error terms at each Nash-Moser iteration step.} and consider so-called effective linear operators
\[
\mathbb{P}'_e(\widehat{U},\hat{\varphi})\dot{U}:=
 L(\widehat{U},\hat{\varphi})\dot{U} +{\cal C}(\widehat{U},\hat{\varphi})
\dot{U}\quad\mbox{and}\quad \mathbb{V}'_e(\hat{\varphi})\dot{V}:= M(\hat{\varphi})\dot{V}.
\]
We write down the final form of our linearized problem for $(\dot{U},\dot{V},\varphi )$:
\begin{align}
 L(\widehat{U},\hat{\varphi})\dot{U} +{\cal C}(\widehat{U},\hat{\varphi})
\dot{U} =f_{\rm I}\qquad \mbox{in}\ \Omega_T,\label{48}\\
 M(\hat{\varphi})\dot{V}  =f_{\rm II}\qquad \mbox{in}\ \Omega_T,\label{49}
\end{align}
\begin{equation}
\left(
\begin{array}{c}
\dot{v}_{N}-\partial_t\varphi -\hat{v}_2\partial_2\varphi-\hat{v}_3\partial_3\varphi +
\varphi\,\partial_1\hat{v}_{N}\\[3pt]
\dot{q}-(\widehat{\mathcal{H}},\dot{\mathcal{H}})+(\widehat{E},\dot{E})+ [\partial_1\hat{q}]\varphi\\[3pt]
\dot{E}_{\tau_2}-
\partial_t(\widehat{\mathcal{H}}_3\varphi ) -\varkappa\dot{\mathcal{H}}_3+\partial_2(\widehat{E}_1\varphi )\\[3pt]
\dot{E}_{\tau_3}+\partial_t(\widehat{\mathcal{H}}_2\varphi ) +\varkappa\dot{\mathcal{H}}_2+\partial_3(\widehat{E}_1\varphi )
\end{array}
\right)=g \qquad \mbox{on}\ \partial\Omega_T,
\label{50}
\end{equation}
\begin{equation}
(\dot{U},\dot{V},\varphi )=0\qquad \mbox{for}\ t<0,\label{51}
\end{equation}
where
\[
\dot{v}_{\rm N}=\dot{v}_1-\dot{v}_2\partial_2\hat{\varphi}-\dot{v}_3\partial_3\hat{\varphi},\quad
\dot{E}_{\tau_j}=\dot{E}_1\partial_j\hat{\varphi}+\dot{E}_j,
\]
\[
[\partial_1\hat{q}]=(\partial_1\hat{q})|_{x^1=0}-(\widehat{\mathcal{H}},\partial_1\widehat{\mathcal{H}})|_{x^1=0}+(\widehat{E},\partial_1\widehat{E})|_{x^1=0},
\]
and $\dot{v}$ and $\dot{q}$ are determined similarly to $v$ and $q$.
We used the second and third equations in \eqref{42} taken at $x^1=0$ while writing down the last two boundary conditions in \eqref{50}. We assume that $f$ and $g$ vanish in the past and consider the case of zero initial data, which is the usual assumption. We postpone the case of nonzero initial data to the nonlinear analysis (construction of a so-called approximate solution; see, e.g., \cite{T09}).

Before formulating the main result for problem \eqref{48}--\eqref{51} we give the definition of the anisotropic weighted Sobolev spaces $H^m_*$. The functional space $H^m_*$ is defined as follows (see \cite{YM,OSY,Sec} and references therein):
$$
H^{m}_*(\mathbb{R}^3_+ ):=\left\{ u\in L_2(\mathbb{R}^3_+) \ | \
\partial^{\alpha}_*\partial_1^k u\in L_2(\mathbb{R}^3_+ )\quad \mbox{if}\quad |\alpha
|+2k\leq m\, \right\},
$$
where $m\in\mathbb{N}$, $\partial^{\alpha}_*=(\sigma \partial_1)^{\alpha_1}\partial_2^{\alpha_2}
\partial_3^{\alpha_3}\,$, and  $\sigma (x^1)\in C^{\infty}(\mathbb{R}_+)$ is
a monotone increasing function such that $\sigma (x^1)=x^1$ in a neighborhood of
the origin and $\sigma (x^1)=1$ for $x^1$ large enough. The space $H^m_*(\mathbb{R}^3_+ )$ is normed by
\[
\|u\|_{m,*}^2= \sum_{|\alpha |+2k\leq m}
\|\partial^{\alpha}_*\partial_1^k u\|_{L_2(\mathbb{R}^3_+)}^2.
\]
We also define the space
\[
H^m_*(\Omega_T)=\bigcap_{k=0}^{m}H^k((-\infty ,T],H^{m-k}_*(\mathbb{R}^3_+ ))
\]
equipped with the norm
\[
[u]^2_{m,*,T}=\int_{-\infty}^T\nt u (t)\nt^2_{m,*}dt,
\quad \mbox{where}\quad
\nt u (t)\nt^2_{m,*}=\sum\limits_{j=0}^m\|\partial_t^ju(t)\|^2_{m-j,*}.
\]
Within this paper we use the space $H^m_*(\Omega_T)$ mainly for $m=1$. Clearly,
the norm for $H^1_*(\Omega_T)$ reads
\[
[u]^2_{1,*,T}=\int_{\Omega_T}\left( u^2 +(\partial_tu)^2 +(\sigma\partial_1u)^2
+(\partial_2u)^2+(\partial_3u)^2\right) dtdx.
\]

We are now in a position to state the main result of this paper.

\begin{theorem}
Let the basic state \eqref{37} satisfies assumptions \eqref{38}--\eqref{44}. Let also
\begin{equation}
{|\widehat{H}_2\widehat{\mathcal{H}}_3-\widehat{H}_3\widehat{\mathcal{H}}_2|
_{|}}_{x^1=0}
\geq \epsilon> 0,\label{52}
\end{equation}
where $\epsilon$ is a fixed constant. Then there exists a positive constant $\mu^*$ such that
for all $(f,g) \in H^3_*(\Omega_T)\times H^3(\partial\Omega_T)$ which vanish in the past and for all $\mu <\mu^*$, where
\begin{equation}
\mu =
{|\widehat{E}_1+\hat{v}_2\widehat{\mathcal{H}}_3-\hat{v}_3\widehat{\mathcal{H}}_2|
_{|}}_{x^1=0},
\label{53}
\end{equation}
the a priori estimate
\begin{equation}
[\dot{U}]_{1,*,T}+\|\dot{V}\|_{H^{1}(\Omega_T)}+\|\varphi\|_{H^1(\partial\Omega_T)}
\leq C\left\{ [f_{\rm I}]_{3,*,T}+\|f_{\rm II}\|_{H^{3}(\Omega_T)}+
\|g\|_{H^{3}(\partial\Omega_T)}\right\}
\label{54}
\end{equation}
holds for problem \eqref{48}--\eqref{51},
where $C=C(K,T)>0$ is a constant independent of the data $(f,g)$.
\label{t1}
\end{theorem}

\section{Reduction to homogeneous boundary conditions and homogeneous Maxwell equations}
\label{sec:5}

Before deriving the basic a priori estimate \eqref{54} we discuss some useful properties of the linearized problem \eqref{48}--\eqref{51} and reduce it to a problem with $f_{\rm II}=0$ and $g=0$. First of all, as for current-vortex sheets \cite{T09}, we can prove the following proposition (see Appendix A in \cite{T09} for its proof for a technically more difficult case when the cut-off function $\chi (x^1)$ is used in the straightening of the front $\Sigma$).

\begin{proposition}
Let the basic state \eqref{37} satisfies assumptions \eqref{38}--\eqref{44}.
Then solutions of problem \eqref{48}--\eqref{51} satisfy
\begin{equation}
{\rm div}\,\dot{\sf h}=f_{15}\quad\mbox{in}\ \Omega_T,
\label{55}
\end{equation}
\begin{equation}
\dot{H}_{N}-\widehat{H}_2\partial_2\varphi -\widehat{H}_3\partial_3\varphi +
\varphi\,\partial_1\widehat{H}_{N}=g_5\quad\mbox{on}\ \partial\Omega_T.
\label{56}
\end{equation}
Here $\dot{\sf h}=(\dot{H}_{N},\dot{H}_2,\dot{H}_3)$, $\dot{H}_{N}=\dot{H}_1-\dot{H}_2\partial_2\hat{\varphi}-\dot{H}_3
\partial_3\hat{\varphi}$,  the new source terms $f_{15}=
f_{15}(t,x )$ and $g_5= g_5(t,x')$, which vanish in the past, are determined by the source terms $f_5$, $f_6$, $f_7$, $g_1$ and the basic state as solutions to the linear inhomogeneous equations
\begin{equation}
\partial_t f_{15}+ (\hat{w} ,\nabla f_{15}) + f_{15}{\rm div}\,\hat{w}={\cal F}\quad\mbox{in}\ \Omega_T,
\label{57}
\end{equation}
\begin{equation}
\partial_t g_5 +\hat{v}_2\partial_2g_5+\hat{v}_3\partial_3g_5+
\left(\partial_2\hat{v}_2+\partial_3\hat{v}_3\right) g_5={\cal G}\quad \mbox{on}\ \partial\Omega_T,
\label{58}
\end{equation}
where
\[
 {\cal F}=\div {f}_{h},\quad {f}_{h}=
(f_{N} ,f_6,f_7),\quad
f_{N}=f_5-f_6\partial_2\hat{\varphi}- f_7\partial_3\hat{\varphi},
\]
\[
{\cal G}=\bigl\{f_{N}+\partial_2\bigl(\widehat{H}_2g_1\bigr)+
\partial_3\bigl(\widehat{H}_3g_1\bigr)\bigr\}\bigr|_{x^1=0}.
\]
\label{p3}
\end{proposition}

For deriving estimate \eqref{54} we will also need a ``linear'' version of Proposition \ref{p2} for the ``vacuum'' part of problem \eqref{48}--\eqref{51}.

\begin{proposition}
Let the basic state \eqref{37} satisfies assumptions \eqref{38}--\eqref{44}.
Then solutions of problem \eqref{48}--\eqref{51} satisfy
\begin{equation}
{\rm div}^-\,\dot{\mathfrak{h}}=f_{16},\quad {\rm div}^-\,\dot{\mathfrak{e}}=f_{17}\quad\mbox{in}\ \Omega_T,
\label{59}
\end{equation}
\begin{equation}
\dot{\mathcal{H}}_{N}-\widehat{\mathcal{H}}_2\partial_2\varphi -\widehat{\mathcal{H}}_3\partial_3\varphi +
\varphi\,\partial_1\widehat{\mathcal{H}}_{N}=g_6\quad\mbox{on}\ \partial\Omega_T.
\label{60}
\end{equation}
Here $\dot{\mathfrak{h}}=(\dot{\mathcal{H}}_{N},\dot{\mathcal{H}}_2,\dot{\mathcal{H}}_3)$, $\dot{\mathcal{H}}_{N}=\dot{\mathcal{H}}_1-\dot{\mathcal{H}}_2\partial_2\hat{\varphi}-\dot{\mathcal{H}}_3\partial_3\hat{\varphi}$, $\dot{\mathfrak{e}}=(\dot{E}_{N},\dot{E}_2,\dot{E}_3)$, $\dot{E}_{N}=\dot{E}_1-\dot{E}_2\partial_2\hat{\varphi}-\dot{E}_3\partial_3\hat{\varphi}$, the new source terms $f_{16}=
f_{16}(t,x )$, $f_{17}=f_{17}(t,x )$ and $g_6= g_6(t,x')$, which vanish in the past, are determined by the source terms $f_9,\ldots ,f_{14}$, $g_3$, $g_4$ and the basic state as solutions to the linear inhomogeneous equations
\begin{equation}
\partial_t f_{16+k}+ \varkappa\,\partial_1f_{16+k}={\cal F}_{1+k}\quad (k=0,1)\qquad\mbox{in}\ \Omega_T
\label{62}
\end{equation}
\begin{equation}
\partial_t g_6={\cal G}_{1}\quad\mbox{on}\ \partial\Omega_T,
\label{63}
\end{equation}
where
\[
\mathcal{F}_1={\rm div}^-\, {f}_{\mathcal{H}},\quad {f}_{\mathcal{H}}=(f_N^1,f_{10},f_{11}),\quad
f_N^1=f_9-f_{10}\partial_2\hat{\varphi}- f_{11}\partial_3\hat{\varphi},
\]
\[
\mathcal{F}_2={\rm div}^-\, {f}_{E},\quad {f}_{E}=(f_N^2,f_{13},f_{14}),
\quad
f_N^2=f_{12}-f_{13}\partial_2\hat{\varphi}- f_{14}\partial_3\hat{\varphi},
\]
\[
\mathcal{G}_1=(f^1_{N}-\partial_2g_4+
\partial_3g_3+f_{16})|_{x^1=0}.
\]
\label{p4}
\end{proposition}

\begin{proof}
Introducing
\[
\dot{\mathfrak{H}}=(\dot{\mathcal{H}}_1,\dot{\mathcal{H}}_{\tau_2},\dot{\mathcal{H}}_{\tau_3})\quad \mbox{and}\quad \dot{\mathfrak{E}}=(\dot{E}_1,\dot{E}_{\tau_2},\dot{E}_{\tau_3})
\]
with $\dot{\mathcal{H}}_{\tau_i}=\dot{\mathcal{H}}_1\partial_i\hat{\varphi} +\dot{\mathcal{H}}_i$,
we rewrite system \eqref{50} as (cf. \eqref{30}, \eqref{31})
\begin{align}
& \partial_t\dot{\mathfrak{h}}+\nabla^-\times \dot{\mathfrak{E}}+\varkappa\,\partial_1\dot{\mathfrak{h}} + B(\hat{\varphi} )\dot{\mathfrak{h}}={f}_{\mathcal{H}},\label{64}  \\
& \partial_t\dot{\mathfrak{e}}-\nabla^-\times \dot{\mathfrak{H}}+\varkappa\,\partial_1\dot{\mathfrak{e}}+ B(\hat{\varphi} )\mathfrak{e}={f}_{E} . \label{65}
\end{align}
Applying the operator ${\rm div}^-$ to \eqref{64} and \eqref{65}, we obtain equations \eqref{62}
for $f_{16}={\rm div}^-\,\dot{\mathfrak{h}}$ and $f_{17}={\rm div}^-\,\dot{\mathfrak{e}}$. Considering the first equation in \eqref{64} at $x^1=0$ and using the last two boundary conditions in \eqref{51} and the first equation in \eqref{59}, we get \eqref{63} for
\begin{equation}
g_6=\dot{\mathcal{H}}_{N}-\partial_2(\widehat{\mathcal{H}}_2\varphi ) -\partial_3(\widehat{\mathcal{H}}_3\varphi ).
\label{66}
\end{equation}
Taking into account the second condition in \eqref{44} on $\partial\Omega_T$, we obtain that the right-hand side of \eqref{66} coincides with the left-hand side of equation \eqref{60}.
\end{proof}

It follows from the first condition in \eqref{40} and assumption \eqref{expan.lin} that the interior equations \eqref{58} and \eqref{62} do not need boundary conditions. Therefore, using standard arguments of the energy method, from \eqref{58} and \eqref{62} we deduce the estimates
\begin{equation}
\|f_{15}\|_{L_2(\Omega_T)}\leq C\|f_{\rm I}\|_{H^1(\Omega_T)}\leq C[f_{\rm I}]_{2,*,T},\label{67}
\end{equation}
\begin{equation}
\|(f_{16},f_{17})\|_{H^2(\Omega_T)}+\|(f_{16},f_{17})|_{x^1=0}\|_{H^2(\partial\Omega_T)}\leq C\|f_{\rm II}\|_{H^3(\Omega_T)},\label{68}
\end{equation}
Here an later on $C$ is a constant that can change from line to line, and sometimes we show the dependence of $C$ from another constants. In particular, in \eqref{67} the constant $C$ depends on $K$ and $T$, and in \eqref{68} it depends on $K$, $T$, and the constant $\delta$ from \eqref{expan.lin}. Using the trace theorem \cite{OSY} for the spaces $H^m_*$ and the usual  trace theorem for  $H^m$, from \eqref{59}, \eqref{63} and \eqref{68} we get
\begin{equation}
\|g_5\|_{H^1(\partial\Omega_T)}\leq C\left\{\|g_1\|_{H^2(\partial\Omega_T)}+\|f_{{\rm I}|x^1=0}\|_{H^1(\partial\Omega_T)}\right\}
\leq C\{\|g_1\|_{H^2(\partial\Omega_T)}+ [f_{\rm I}]_{2,*,T}\}.\label{69}
\end{equation}
\begin{multline}
\|g_6\|_{H^2(\partial\Omega_T)}\leq C\left\{\|(g_3,g_4)\|_{H^3(\partial\Omega_T)}+
\|(f_{\rm II},f_{16})|_{x^1=0}\|_{H^2(\partial\Omega_T)}\right\}\\
\leq C\{\|(g_3, g_4)\|_{H^3(\partial\Omega_T)}+ \|f_{\rm II}\|_{H^3(\Omega_T)}.\label{70}
\end{multline}

Below we reduce problem \eqref{48}--\eqref{51} to that with $(f_{\rm II},f_{15},f_{16},f_{17})=0$ and $(g,g_5,g_6)=0$, and in Sections \ref{sec:6} and \ref{sec:8} we will work with a reduced problem. If, however, we considered the original problem \eqref{48}--\eqref{51}, taking into account \eqref{67}--\eqref{70}, in the basic a priori estimate we should lose one derivative from $f_{\rm I}$ (in terms of $H^m_*$ norms), one derivative from $g_{\rm I}=(g_1,g_2)$ and two derivatives from the source terms $f_{\rm II}$ and $g_{\rm II}=(g_3, g_4)$ to the solution.\footnote{The fact itself that we lose derivatives from the source terms to the solution is quite natural because the constant coefficients version of problem \eqref{48}--\eqref{51} can satisfy the Kreiss-Lopatinski condition  only in a weak sense. We prove this in Section \ref{sec:7} at least for particular cases.} That is, the a priori estimate \eqref{55} is a roughened version of the basic a priori estimate. We prefer to sacrifice the lesser number of lost derivatives for obtaining such a reduced problem with homogeneous boundary conditions and $f_{\rm II}=0$ that for this problem we can derive an a priori estimate with {\it no loss of derivatives} from a new source  term $f_{\rm I}$. In the future this will enable us to prove the existence of solutions of the reduced linearized problem by an almost classical argument (see \cite{LP}).

We first reduce  problem \eqref{48}--\eqref{51} to that with $f_{\rm II}=0$ and $g_{\rm II}=0$.
To this end we consider the auxiliary unknown $\widetilde{V}$ satisfying system \eqref{49}, the last two boundary conditions in \eqref{50} with $\varphi =0$ and zero initial data:
\begin{equation}
\partial_t\widetilde{V}-\widehat{B}_1\partial_1\widetilde{V}+{B}_2\partial_2\widetilde{V}+
{B}_3\partial_3\widetilde{V}=f_{\rm II} \qquad \mbox{in}\ \Omega_T,\label{71}
\end{equation}
\begin{equation}
\widetilde{E}_{\tau_2}-\varkappa\widetilde{\mathcal{H}}_3 =g_3,\quad
\widetilde{E}_{\tau_3}+\varkappa\widetilde{\mathcal{H}}_2=g_4 \qquad \mbox{on}\ \partial\Omega_T,
\label{72}
\end{equation}
\begin{equation}
\widetilde{V}=0\qquad \mbox{for}\ t<0,\label{73}
\end{equation}
where
\[
\widehat{B}_1=:\widetilde{B}_1(\hat{\varphi})=\left(
\begin{array}{cccccc}
-\varkappa & 0 & 0& 0 & \partial_3\hat{\varphi} & -\partial_2\hat{\varphi} \\
0 & -\varkappa & 0& -\partial_3\hat{\varphi} & 0 & -1 \\
0 & 0 & -\varkappa& \partial_2\hat{\varphi} & 1 & 0 \\
0 & -\partial_3\hat{\varphi} & \partial_2\hat{\varphi}& -\varkappa & 0 & 0 \\
\partial_3\hat{\varphi} & 0 & 1& 0 & -\varkappa & 0 \\
-\partial_2\hat{\varphi} & -1 & 0& 0 & 0 & -\varkappa
\end{array}
\right)
\]
and
\begin{equation}
(\widehat{B}_1\widetilde{V},\widetilde{V})=\widetilde{\mathcal{H}}_3\widetilde{E}_{\tau_2}-\widetilde{\mathcal{H}}_2\widetilde{E}_{\tau_3}+\widetilde{\mathcal{H}}_1\bigl(\widetilde{E}_{\tau_2}\partial_3\hat{\varphi}-\widetilde{E}_{\tau_3}\partial_2\hat{\varphi}\bigr)-\varkappa |\widetilde{V}|^2.
\label{quadr}
\end{equation}
Clearly,  $\tilde{\mathfrak{h}}$ and $\tilde{\mathfrak{e}}$ satisfy equations
\eqref{59}:
\begin{equation}
{\rm div}^-\,\tilde{\mathfrak{h}}=f_{16},\quad {\rm div}^-\,\tilde{\mathfrak{e}}=f_{17}\quad\mbox{in}\ \Omega_T.\label{74}
\end{equation}

\begin{proposition}
Problem \eqref{71}--\eqref{73} with sufficiently smooth coefficients $\varkappa$, $\partial_2\hat{\varphi}$ and $\partial_3\hat{\varphi}$ has a unique solution $\widetilde{V}\in H^s(\Omega_T)$, with $s\geq 0$, for all $(f_{\rm II},g_{\rm II})\in H^{s+1/2}(\Omega_T)\times H^{s+1/2}(\partial\Omega_T)$. This solution obeys the estimate
\begin{equation}
\|\widetilde{V}\|_{H^s(\Omega_T)}\leq C\left\{\|f_{\rm II}\|_{H^{s+1/2}(\Omega_T)} +\|g_{\rm II}\|_{H^{s+1/2}(\partial\Omega_T)}\right\},
\label{est.}
\end{equation}
where $C >0$ is a constants independent on the data $(f_{\rm II},g_{\rm II})$ and depending on the coefficients\,\footnote{Below we use Proposition \ref{p5} with an index $s$ for which assumption \eqref{38} on the coefficients is enough.} and the time $T$.
\label{p5}
\end{proposition}

\begin{proof}
We note that the number of boundary conditions in \eqref{72} coincides with the number of incoming characteristics of system \eqref{71}, i.e., problem \eqref{71}--\eqref{73} has the property of maximality \cite{LP}. To prove the existence of solutions to problem \eqref{71}--\eqref{73} we reduce it to a problem with homogeneous and {\it maximally dissipative} boundary conditions. We first pass to a new unknown for which equations \eqref{74} become homogeneous (with $f_{16}=f_{17}=0$). Actually, we even make system \eqref{71} homogeneous. Let the vector-function $\check{V}=(\check{\mathcal{H}},\check{E})$ which vanishes in the past satisfies system \eqref{71} and the boundary conditions
\begin{equation}
\check{E}_{\tau_2}=0,\quad
\check{E}_{\tau_3}=0 \qquad \mbox{on}\ \partial\Omega_T.
\label{77}
\end{equation}
Then, in view of \eqref{quadr}, \eqref{77} and \eqref{expan.lin}, we have
\[
(\widehat{B}_1\check{V},\check{V})|_{x^1=0}=-\varkappa |\check{V}|^2\geq \delta |\check{V}|^2.
\]
That is, the boundary conditions \eqref{77} are strictly (and maximally) dissipative and, hence, there exists a unique solution $\check{V}$ obeying the estimate (see \cite{RM,Sec98})
\begin{equation}
\|\check{V}\|_{H^s(\Omega_T)}+\|\check{V}_{|x^1=0}\|_{H^s(\partial\Omega_T)}\leq C\|f_{\rm II}\|_{H^s(\Omega_T)}.
\label{78}
\end{equation}

The new unknown ${V}^{\sharp}=\widetilde{V} -\check{V}$
satisfies problem \eqref{71}--\eqref{73} with $f_{\rm II}=0$ and $g_3$ and $g_4$ replaced by
\[
\check{g}_3=g_3+\varkappa\check{\mathcal{H}}_3 \quad\mbox{and}\quad
\check{g}_4=g_4-\varkappa\check{\mathcal{H}}_2.
\]
respectively. It follows from \eqref{78} that
\begin{equation}
\|(\check{g}_3,\check{g}_4)\|_{H^s(\Omega_T)}\leq C\left\{\|f_{\rm II}\|_{H^{s}(\Omega_T)}
+\|{g}_{\rm II}\|_{H^s(\partial\Omega_T)}\right\}.
\label{79}
\end{equation}
Moreover, the new unknown ${V}^{\sharp}$ satisfies the homogeneous equations \eqref{74}:
\begin{equation}
{\rm div}^-\,{\mathfrak{h}}^{\sharp}=0,\quad {\rm div}^-\,{\mathfrak{e}}^{\sharp}=0\quad\mbox{in}\ \Omega_T.\label{80}
\end{equation}

Using the secondary symmetrization \eqref{20} of the Maxwell equations described in Section \ref{sec:2}, we now reduce the problem for ${V}^{\sharp}$ to an equivalent problem with dissipative (but not strictly dissipative) boundary conditions. Taking into account \eqref{20}, \eqref{22} and using \eqref{80} and the fact that
\[
R_1{\rm div}^-\,{\mathfrak{h}}^{\sharp}+R_2{\rm div}^-\,{\mathfrak{e}}^{\sharp}= \left(
\mathcal{K}_2\partial_2\hat{\varphi}+\mathcal{K}_3\partial_3\hat{\varphi}-\mathcal{K}_1\right)\partial_1V^{\sharp}+\mathcal{K}_2\partial_2V^{\sharp}+\mathcal{K}_3\partial_3V^{\sharp},
\]
from system \eqref{71} for ${V}^{\sharp}$ (with $f_{\rm II}=0$) we obtain
\begin{equation}
\mathcal{B}_0\partial_t{V}^{\sharp}-\widehat{\mathcal{B}}_1\partial_1{V}^{\sharp}+\mathcal{B}_2\partial_2{V}^{\sharp}+
\mathcal{B}_3\partial_3{V}^{\sharp}=0 \qquad \mbox{in}\ \Omega_T,\label{81}
\end{equation}
where $\widehat{\mathcal{B}}_1:=\widetilde{\mathcal{B}}_1(\hat{\varphi})=\mathcal{B}_1 -\varkappa \mathcal{B}_1-\partial_2\hat{\varphi}\mathcal{B}_2 -\partial_3\hat{\varphi}\mathcal{B}_3.$ We choose
\[
\nu_1 = \varkappa ,\quad  \nu_2=\nu_3 =0.
\]
Then
\[
\widehat{\mathcal{B}}_1
=\left(
\begin{array}{cccccc}
0& -\varkappa\partial_2\hat{\varphi}  & -\varkappa\partial_3\hat{\varphi} & 0 & \partial_3\hat{\varphi} & -\partial_2\hat{\varphi} \\
-\varkappa\partial_2\hat{\varphi} & -2\varkappa & 0& -\partial_3\hat{\varphi} & 0 & -1-\varkappa^2 \\
-\varkappa\partial_3\hat{\varphi} & 0 & -2\varkappa & \partial_2\hat{\varphi} & 1 +\varkappa^2& 0 \\
0 & -\partial_3\hat{\varphi} & \partial_2\hat{\varphi}& 0& -\varkappa\partial_2\hat{\varphi} & -\varkappa\partial_3\hat{\varphi} \\
\partial_3\hat{\varphi} & 0 & 1+\varkappa^2& -\varkappa\partial_2\hat{\varphi} & -2\varkappa & 0 \\
-\partial_2\hat{\varphi} & -1-\varkappa^2 & 0& -\varkappa\partial_3\hat{\varphi} & 0 & -2\varkappa
\end{array}
\right).
\]
Since $|\varkappa |<1$ the hyperbolicity condition \eqref{21} holds, i.e., $\mathcal{B}_0>0$. It is noteworthy that from system \eqref{81} we can derive equations \eqref{80} and using them obtain the original system \eqref{71}. That is, systems \eqref{71} and  \eqref{81} are equivalent and we may use \eqref{81} to prove the existence of a solution ${V}^{\sharp}$.

Assuming for a moment that $\check{g}_3=\check{g}_4=0$ and using the homogeneous boundary conditions for ${V}^{\sharp}$,
\[
{E}^{\sharp}_{\tau_2}=\varkappa{\mathcal{H}}^{\sharp}_3 ,\quad
{E}^{\sharp}_{\tau_3}=-\varkappa{\mathcal{H}}^{\sharp}_2 \qquad \mbox{on}\ \partial\Omega_T,
\]
after straightforward calculations we get
\[
(\widehat{\mathcal{B}}_1{V}^{\sharp},{V}^{\sharp})|_{x^1=0}=0.
\]
This means that the boundary conditions are dissipative (but not strictly dissipative). To prove the existence of a solution ${V}^{\sharp}$ we reduce the problem for it to that with homogeneous boundary conditions (with $\check{g}_3=\check{g}_4=0$). To this end we use the classical argument \cite{RM} suggesting to subtract from the solution a more regular function satisfying the boundary conditions. Thus, referring to \cite{RM}, we can conclude that there exists a unique solution ${V}^{\sharp}$, and this solution obeys the estimate
\begin{equation}
\|V^{\sharp}\|_{H^m(\Omega_T)}\leq C\|(\check{g}_3,\check{g}_4)\|_{H^{m+1/2}(\partial\Omega_T)}.
\label{82}
\end{equation}
Hence, there exists a unique solution $\widetilde{V}\in H^s(\Omega_T)$ to problem \eqref{71}--\eqref{73} and, in view of \eqref{78}, \eqref{79} and \eqref{82}, it obeys estimate \eqref{est.}.
\end{proof}

Having in hand Proposition \ref{p5}, in problem \eqref{48}--\eqref{51} we can pass to the new ``vacuum'' unknown
\begin{equation}
V=\dot{V}-\widetilde{V}.
\label{83}
\end{equation}
Here to simplify the notations we again use the same letter for the new ``vacuum'' unknown as for the ``vacuum'' unknown in the linearized problem before passing to the ``good unknown'' $\dot{V}$ in \eqref{45}. The new unknown $V$ satisfies the homogeneous system \eqref{49} (with $f_{\rm II}=0$) and the homogeneous last two boundary conditions in \eqref{50} (with $g_{\rm II}=0$).

Concerning the ``plasma'' part of \eqref{48}--\eqref{51}, the change of unknown \eqref{83} only changes the source term in the second boundary condition in \eqref{50}:
\[
\dot{q}-(\widehat{\mathcal{H}},{\mathcal{H}})+(\widehat{E},{E})+ [\partial_1\hat{q}]\varphi =
\check{g}_2:=g_2+(\widehat{\mathcal{H}},\widetilde{\mathcal{H}})-(\widehat{E},\widetilde{E})\qquad \mbox{on}\ \partial\Omega_T.
\]
By virtue of \eqref{est.} and the trace theorem in $H^m$, we have the estimate
\begin{multline}
\|\check{g}_2\|_{H^m(\Omega_T)}\leq C\left\{\|{g}_2\|_{H^m(\partial\Omega_T)}+
\|\widetilde{V}_{|x^1=0}\|_{H^{m}(\partial\Omega_T)}\right\} \\
\leq C\left\{\|{g}_2\|_{H^m(\partial\Omega_T)}+
\|\widetilde{V}\|_{H^{m+1/2}(\Omega_T)}\right\} \\
\leq C\left\{\|{g}\|_{H^{m+1}(\partial\Omega_T)}+
\|f_{\rm II}\|_{H^{m+1}(\Omega_T)} \right\}.
\label{84}
\end{multline}

To reduce the ``plasma'' part of the linearized problem to that with homogeneous boundary conditions (with $g_1=\check{g}_2=0$) and the homogeneous equations \eqref{55} and \eqref{56}
(with $f_{15}=g_5=0$) we can exploit the same arguments as in \cite{T09} for current-vortex sheets. Here we even describe a ``shifting'' function in an almost concrete form:
\[
\widetilde{U} = (\tilde{p},\widehat{\Gamma}^3\tilde{v}_1,0,0,\widetilde{H},0),\qquad
\tilde{v}_1=\chi (x^1)g_1,\quad \tilde{p}= \chi (x^1)\check{g}_2 -\frac{1}{2}\,\delta \widetilde{B}^2,
\]
\[
\delta \widetilde{B}^2=\frac{2}{\widehat{\Gamma}}
\left( (\hat{b},\widetilde{H}) +(\hat{v},\widehat{H})(\widehat{H},\tilde{u})-\widehat{B}^2(\hat{v},\tilde{u})\right),\quad
\tilde{u}=\widehat{\Gamma}^3\tilde{v},\quad \tilde{v}=(\tilde{v}_1,0,0),
\]
where the cut-off function $\chi\in C^{\infty}_0(\mathbb{R_+})$ equals to 1 on $[0,1]$ and
$\widetilde{H}$ solves the equation for $\dot{H}$ contained in (48) with $\dot{v}=\tilde{v}=(\tilde{v}_1,0,0)$:
\begin{equation}
\partial_t\widetilde{H}+ (\hat{w} ,\nabla )
\widetilde{H} - (\tilde{\sf h} ,\nabla ) \hat{v} + \widetilde{H}{\rm div}\,\hat{w}
 ={f}_{H} +(\hat{\sf h} ,\nabla ) \tilde{v} - \widehat{H}\partial_1\tilde{v}_1-\tilde{v}_1\partial_1\widehat{H}\qquad \mbox{in}\ \Omega_T,
 \label{85}
\end{equation}
where $\tilde{\sf h}=(\widetilde{H}_1-\widetilde{H}_2\partial_2\hat{\varphi}
-\widetilde{H}_3\partial_3\hat{\varphi},\widetilde{H}_2,\widetilde{H}_3)$ and $f_H=(f_5,f_6,f_7)$. It is very important that $\hat{w}|_{x^1=0}=0$, i.e., the linear equation \eqref{85} does not need boundary conditions and we easily get the estimate
\begin{equation}
[\partial_1\widetilde{H}]_{s,*,T}+[\widetilde{H}]_{s+1,*,T}\leq C\left\{ [f_{\rm I}]_{s+2,*,T}+\|g_1\|_{H^{s+2}(\partial\Omega_T)}\right\}.
\label{86}
\end{equation}

It is clear that $\widetilde{U}$ satisfies the first two boundary conditions in \eqref{50} with $g_2=\check{g}_2$. Then the new unknown $U=\dot{U}-\widetilde{U}$ satisfies the homogeneous boundary conditions \eqref{50} and system \eqref{48} with $f_{\rm I}=F$, where
\[
F =(F_1,\ldots, F_8)=f_{\rm I}-\mathbb{P}'_e(\widehat{U},\widehat{\Psi}^+)\widetilde{U}.
\]
In view of \eqref{85}, $F_H=(F_5,F_6,F_7)=0$. Since the boundary conditions became homogeneous and $F_H=0$, it follows from Proposition \ref{p3} that the new ``plasma'' unknown $U$ satisfies \eqref{55} and \eqref{56} with $f_{15}=g_5=0$. Moreover, taking into account \eqref{84} and \eqref{86}, for the new source term $F$ we have the estimate
\begin{multline}
[F]_{s,*,T}\leq C\Bigl\{[f_{\rm I}]_{s,*,T}+\|(g_1,\check{g}_2)\|_{H^{s+1}(\partial\Omega_T)}\\+ [\partial_1\widetilde{H}]_{s,*,T}+[\widetilde{H}]_{s+1,*,T}\Bigr\}\\
 \leq C\left\{ [f_{\rm I}]_{s+2,*,T}+\|f_{\rm II}\|_{H^{s+2}(\Omega_T)}+\|g\|_{H^{s+2}(\partial\Omega_T)}\right\}.
 \label{87}
\end{multline}
In this paper we need \eqref{87} for $s=1$:
\begin{equation}
[F]_{1,*,T}
 \leq C\left\{ [f_{\rm I}]_{3,*,T}+\|f_{\rm II}\|_{H^{3}(\Omega_T)}+\|g\|_{H^{3}(\partial\Omega_T)}\right\}.
 \label{88}
\end{equation}

We write down the final form of our reduced linearized problem for the new unknown $({U},{V})$ and the interface perturbation $\varphi$:
\begin{align}
& \widehat{A}_0\partial_t{U}+\sum_{j=1}^{3}\widehat{A}_j\partial_j{U}+
\widehat{\cal C}{U}=F \qquad \mbox{in}\ \Omega_T,\label{89}
\\
& \partial_t{V}-\widehat{B}_1\partial_1{V}+{B}_2\partial_2{V}+
{B}_3\partial_3{V}=0 \qquad \mbox{in}\ \Omega_T,\label{90}
\end{align}
\begin{equation}
\left\{ \begin{array}{ll}
\partial_t\varphi={v}_{N}-\hat{v}_2\partial_2\varphi-\hat{v}_3\partial_3\varphi +
\varphi\,\partial_1\hat{v}_{N}, & \\[6pt]
{q}=(\widehat{\mathcal{H}},{\mathcal{H}})-(\widehat{E},{E})-  [ \partial_1\hat{q}] \varphi  , & \\[6pt]
{E}_{\tau_2}=
\partial_t(\widehat{\mathcal{H}}_3\varphi )
-\partial_2(\widehat{E}_1\varphi )+\varkappa{\mathcal{H}}_3, & \\[6pt]
{E}_{\tau_3}=-\partial_t(\widehat{\mathcal{H}}_2\varphi )
-\partial_3(\widehat{E}_1\varphi ) -\varkappa{\mathcal{H}}_2,& \quad \mbox{on}\ \partial\Omega_T,
\end{array}\right.\label{91}
\end{equation}
\begin{equation}
({U},{V},\varphi )=0\qquad \mbox{for}\ t<0,\label{92}
\end{equation}
where $\widehat{A}_{\alpha}:={A}_{\alpha}(\widehat{U})$ $(\alpha =0,2,3)$,
$\widehat{A}_1:=\widetilde{A}_1(\widehat{U},\hat{\varphi})$,
$\widehat{\cal C}:={\cal C}(\widehat{U},\hat{\varphi})$. Recall also that solutions to problem \eqref{89}--\eqref{92} have the following very important property. Since $F_H=0$ and equations \eqref{90} and the boundary conditions \eqref{91} are homogeneous, solutions to problem \eqref{89}--\eqref{92} satisfy
\begin{equation}
{\rm div}\,{\sf h}=0\qquad\mbox{in}\ \Omega_T,
\label{93}
\end{equation}
\begin{equation}
{\rm div}^-\,{\mathfrak{h}}=0,\quad {\rm div}^-\,{\mathfrak{e}}=0\qquad\mbox{in}\ \Omega_T,
\label{94}
\end{equation}
\begin{equation}
{H}_{N}=\widehat{H}_2\partial_2\varphi +\widehat{H}_3\partial_3\varphi -
\varphi\,\partial_1\widehat{H}_{N}\qquad\mbox{on}\ \partial\Omega_T,
\label{95}
\end{equation}
\begin{equation}
{\mathcal{H}}_{N}=\widehat{\mathcal{H}}_2\partial_2\varphi +\widehat{\mathcal{H}}_3\partial_3\varphi -
\varphi\,\partial_1\widehat{\mathcal{H}}_{N}\qquad\mbox{on}\ \partial\Omega_T.
\label{96}
\end{equation}
We recall that
\[
v_{N}= v_1-v_2\partial_2\hat{\varphi}-v_3\partial_3\hat{\varphi},\quad
E_{\tau_j}=E_1\partial_j\hat{\varphi}+E_j,\quad
{\sf h}=({H}_{N},{H}_2,{H}_3),
\]
\[
{H}_{N}={H}_1-{H}_2\partial_2\hat{\varphi}-{H}_3
\partial_3\hat{\varphi},\quad
{\mathcal{H}}_{N}={\mathcal{H}}_1-{\mathcal{H}}_2\partial_2\hat{\varphi}-{\mathcal{H}}_3\partial_3\hat{\varphi},
\]
\[
{\mathfrak{h}}=({\mathcal{H}}_{N},{\mathcal{H}}_2,{\mathcal{H}}_3)
,\quad {\mathfrak{e}}=({E}_{N},{E}_2,{E}_3),\quad {E}_{N}={E}_1-{E}_2\partial_2\hat{\varphi}-{E}_3\partial_3\hat{\varphi},
\]
and below we will also use the notations
\[
{\mathfrak{H}}=({\mathcal{H}}_1,{\mathcal{H}}_{\tau_2},{\mathcal{H}}_{\tau_3}),\quad {\mathcal{H}}_{\tau_i}={\mathcal{H}}_1\partial_i\hat{\varphi} +{\mathcal{H}}_i,\quad  {\mathfrak{E}}=({E}_1,{E}_{\tau_2},{E}_{\tau_3}).
\]

From now on we concentrate on the proof of the following basic a priori estimate for solutions to problem \eqref{89}--\eqref{92}:
\begin{equation}
[{U}]_{1,*,T}+\|{V}\|_{H^{1}(\Omega_T)}+\|\varphi\|_{H^1(\partial\Omega_T)}
\leq C [F]_{1,*,T},
\label{97}
\end{equation}
where $C=C(K,T)>0$ is a constant independent of $F$. Taking into account \eqref{88}, \eqref{est.} and a corresponding estimate for $\widetilde{U}$, from \eqref{97} we obtain the desired a priori estimate \eqref{54} for the original unknown $(\dot{U},\dot{V})$ and the interface perturbation $\varphi$. It is worth noting that estimate \eqref{97} is with no loss of derivatives from the source term $F$ to the solution. This will be useful for the proof of the existence of solutions to problem \eqref{89}--\eqref{92} (this work is postponed to the future).

\section{Energy a priori estimate for the constant coefficients problem}
\label{sec:6}

To clarify main ideas in a clearer form and to avoid unimportant technical difficulties we first prove estimate \eqref{97} for the case of constant coefficients of \eqref{89}--\eqref{92} and dropped zero-order terms. In Section \ref{sec:8} we will extend this result to variable coefficients. That is, we now ``freeze'' coefficients of problem \eqref{89}--\eqref{92} and drop zero-order terms. Moreover, we assume that $\partial_2\hat{\varphi}=\partial_3\hat{\varphi}$. Then we obtain a constant coefficients linear problem which is, in fact, the result of the linearization of the original nonlinear free boundary value problem \eqref{8}, \eqref{13}, \eqref{14}, \eqref{16}--\eqref{18} about its {\it exact} constant solution
\begin{align}
& U=\widehat{U}=(\hat{p},\hat{u},\widehat{H},\widehat{S})={\rm const}\qquad \mbox{for}\ x^1>\varkappa t,\nonumber \\
& V=\widehat{V}=(\widehat{\mathcal{H}},\widehat{E})={\rm const}\qquad \mbox{for}\ x^1<\varkappa t\nonumber
\end{align}
for the planar plasma-vacuum interface $x^1 =\varkappa t$, where $\varkappa$, with $-1<\varkappa < 0$, is a constant interface speed. We know that the components of this exact constant solution satisfy the relations (cf. \eqref{40}, \eqref{43})
\begin{equation}
\hat{v}_1=\hat{u}_1/\widehat{\Gamma}=\varkappa ,\quad \widehat{H}_1=\widehat{\mathcal{H}}_1=0,\quad \widehat{E}_2=\varkappa\widehat{\mathcal{H}}_3,\quad
\widehat{E}_3=-\varkappa\widehat{\mathcal{H}}_2.
\label{98}
\end{equation}

Taking into account \eqref{98}, we have the following constant coefficients problem:
\begin{align}
& \widehat{A}_0\partial_t{U}+\sum_{j=1}^{3}\widehat{A}_j\partial_j{U}=F \qquad \mbox{in}\ \Omega_T,\label{99}
\\
& \partial_t{V}-\widehat{B}_1\partial_1{V}+{B}_2\partial_2{V}+
{B}_3\partial_3{V}=0 \qquad \mbox{in}\ \Omega_T,\label{100}
\end{align}
\begin{equation}
\left\{ \begin{array}{ll}
\partial_t\varphi={v}_1-\hat{v}_2\partial_2\varphi-\hat{v}_3\partial_3\varphi , & \\[6pt]
{q}=\widehat{\mathcal{H}}_2({\mathcal{H}}_2+\varkappa E_3) +\widehat{\mathcal{H}}_3({\mathcal{H}}_3-\varkappa E_2)-\widehat{E}_1{E}_1  , & \\[6pt]
{E}_{2}=\widehat{\mathcal{H}}_3\partial_t \varphi
-\widehat{E}_1\partial_2 \varphi +\varkappa{\mathcal{H}}_3, & \\[6pt]
{E}_{3}=-\widehat{\mathcal{H}}_2\partial_t\varphi
-\widehat{E}_1\partial_3 \varphi  -\varkappa{\mathcal{H}}_2,& \quad \mbox{on}\ \partial\Omega_T,
\end{array}\right.\label{101}
\end{equation}
\begin{equation}
({U},{V},\varphi )=0\qquad \mbox{for}\ t<0,\label{102}
\end{equation}
where $\widehat{A}_{\alpha}:={A}_{\alpha}(\widehat{U})$ ($\alpha =0,2,3)$,
$\widehat{A}_1:={A}_1(\widehat{U})-\varkappa \widehat{A}_0$, $\widehat{B}_1:=B_1-\varkappa I$, $v=\left(u-(\hat{v},u)\hat{v}\right)/\widehat{\Gamma}$, etc. Moreover, solutions to problem \eqref{99}--\eqref{102} satisfy
\begin{equation}
{\rm div}\,H=0,\quad {\rm div}^-\,{\mathcal{H}}=0,\quad {\rm div}^-\,E=0\qquad\mbox{in}\ \Omega_T,
\label{103}
\end{equation}
\begin{equation}
{H}_{1}=\widehat{H}_2\partial_2\varphi +\widehat{H}_3\partial_3\varphi ,\quad
{\mathcal{H}}_{1}=\widehat{\mathcal{H}}_2\partial_2\varphi +\widehat{\mathcal{H}}_3\partial_3\varphi \qquad\mbox{on}\ \partial\Omega_T.
\label{104}
\end{equation}
From the physical point of view, the well-posedness of problem \eqref{99}--\eqref{102} can be interpreted as the stability of a planar relativistic plasma-vacuum interface (or the  macroscopic stability of a corresponding nonplanar interface).

For problem \eqref{99}--\eqref{102} we now derive the a priori estimate \eqref{97} under the {\it sufficient} stability condition that the constant
$\mu = |\hat{\mu}|$, with
\[
\hat{\mu}= \widehat{E}_1+\hat{v}_2\widehat{\mathcal{H}}_3-\hat{v}_3\widehat{\mathcal{H}}_2,
\]
is small enough (cf.  \eqref{53}) and the vectors $(\widehat{H}_2,\widehat{H}_3)$ and $(\widehat{\mathcal{H}}_2,\widehat{\mathcal{H}}_3)$ are nonzero and nonparallel to each other (cf. \eqref{52}), i.e.,
\begin{equation}
\widehat{H}_2\widehat{\mathcal{H}}_3-\widehat{H}_3\widehat{\mathcal{H}}_2\neq 0.\label{52a}
\end{equation}

Using the secondary symmetrization \eqref{20} of the Maxwell equations, as in \eqref{81}, we
equivalently rewrite system \eqref{100} as
\begin{equation}
\mathcal{B}_0\partial_t{V}-\widehat{\mathcal{B}}_1\partial_1{V}+\mathcal{B}_2\partial_2{V}+
\mathcal{B}_3\partial_3{V}=0 \qquad \mbox{in}\ \Omega_T,\label{105}
\end{equation}
where $\widehat{\mathcal{B}}_1=\mathcal{B}_1 -\varkappa \mathcal{B}_0.$ Now our choice of $\nu_i$ is the following:
\begin{equation}
\nu =\hat{v}= (\hat{v}_1 ,\hat{v}_2,\hat{v}_3)
\label{106}
\end{equation}
($\hat{v}_1=\varkappa$, see \eqref{98}). Since $|\hat{v}|<1$ the hyperbolicity condition \eqref{21} holds. By standard arguments of the energy method applied to the symmetric hyperbolic systems \eqref{99} (multiplied by $1/\widehat{\Gamma}$) and \eqref{105} we obtain
\begin{equation}
I(t)+ 2\int_{\partial\Omega_t}\mathcal{Q}\,{\rm d}x'{\rm d}s\leq C
\left( \| F\|^2_{L_2(\Omega_T)} +\int_0^tI(s)\,{\rm d}s\right),
\label{107}
\end{equation}
where
\[
I(t)=\int_{\mathbb{R}^3_+}\Bigl(\,\frac{1}{\widehat{\Gamma}}(\widehat{A}_0U,U)+(\mathcal{B}_0V,V)\Bigr){\rm d}x,
\quad
\mathcal{Q}=\frac{1}{2}(\widehat{\mathcal{B}}_1V,V)|_{x^1=0}-\frac{1}{2\widehat{\Gamma}}(\widehat{A}_1U,U)|_{x^1=0}.
\]

Taking into account \eqref{35}, we write down the quadratic form $\mathcal{Q}$:
\begin{align*}
\mathcal{Q}= &\,\bigl\{-\varkappa(\mathcal{H}_2^2+\mathcal{H}_3^2+ E_2^2+ E_3^2)+\mathcal{H}_1(\hat{v}_2\mathcal{H}_2+\hat{v}_3\mathcal{H}_3)
+E_1(\hat{v}_2E_2+\hat{v}_3E_3)
\\
&\;\;
+\varkappa \mathcal{H}_1(\hat{v}_2E_3-\hat{v}_3E_2)
+(1+\varkappa^2)(\mathcal{H}_3E_2-\mathcal{H}_2E_3) -qv_1\bigr\}\bigr|_{x^1=0}\,.
\end{align*}
Using the boundary conditions \eqref{101} and the second condition in \eqref{104}, after relatively long calculations, which are omitted here, we get
\[
\mathcal{Q}=\hat{\mu}\left.\left\{ E_1\partial_t\varphi +(\mathcal{H}_2+\varkappa E_3)\partial_3\varphi -(\mathcal{H}_3-\varkappa E_2)\partial_2\varphi \right\}\right|_{x^1=0}
\]
or
\[
\mathcal{Q}=\partial_t\left(\hat{\mu}\varphi E_1|_{x^1=0} \right)+\partial_2\left(\hat{\mu}\varphi(\varkappa E_2-\mathcal{H}_3)|_{x^1=0} \right)  + \partial_3\left(\hat{\mu}\varphi (\mathcal{H}_2+\varkappa E_3)|_{x^1=0}\right) +\mathcal{Q}_0,
\]
and in view of the fourth equation in \eqref{101} and the third equation in \eqref{104},
\[
\mathcal{Q}_0=-\hat{\mu}\varphi \left.\left( \partial_tE_1+\varkappa\partial_1E_1+\partial_3\mathcal{H}_2-
\partial_2\mathcal{H}_3 \right)\right|_{x^1=0}=0.
\]

Then, it follows from \eqref{107} that
\begin{equation}
I(t)+ 2\int_{\mathbb{R}^2}\hat{\mu}\varphi E_1|_{x^1=0}{\rm d}x'\leq C
\left( \| F\|^2_{L_2(\Omega_T)} +\int_0^tI(s)\,{\rm d}s\right).
\label{109}
\end{equation}
At first sight inequality \eqref{109} is absolutely useless. Indeed, we cannot apply it for deriving an $L_2$ a priori estimate. On the other hand, if we differentiate systems \eqref{99} and \eqref{105} with respect to $x^2$, $x^3$ and $t$, we easily obtain the following counterparts of \eqref{109} for the first-order tangential derivatives of $U$ and $V$:
\begin{equation}
I_{\alpha}(t)+ 2\int_{\mathbb{R}^2}\hat{\mu}\,\partial_{\alpha}\varphi \, \partial_{\alpha}E_1|_{x^1=0}{\rm d}x'\leq C
\left( [F]^2_{1,*T} +\int_0^tI_{\alpha}(s)\,{\rm d}s\right),
\label{110}
\end{equation}
where $\alpha =0,2,3$,
\[
I_{\alpha}(t)=\int_{\mathbb{R}^3_+}\Bigl(\,\frac{1}{\widehat{\Gamma}}(\widehat{A}_0\partial_{\alpha}U,\partial_{\alpha}U)+(\mathcal{B}_0\partial_{\alpha}V,\partial_{\alpha}V)\Bigr){\rm d}x\quad \mbox{and}\quad \partial_0:=\partial_t.
\]
The terms $\partial_{\alpha}\varphi \, \partial_{\alpha}E_1|_{x^1=0}$ appearing in the boundary integral are actually lower-order terms and below we explain how to treat them by using our important assumption \eqref{52a} and passing to the volume integral.

Thanks to  \eqref{52a}, from \eqref{104} we get
\begin{equation}
\left\{
\begin{array}{l}
{\displaystyle
\partial_2\varphi =\frac{(\widehat{\mathcal{H}}_3{H}_1-\widehat{H}_3{\mathcal{H}}_1)|_{x^1=0}
}{\widehat{H}_2\widehat{\mathcal{H}}_3-\widehat{H}_3\widehat{\mathcal{H}}_2}=
a_1^1{H}_1|_{x^1=0}+a_2^1{\mathcal{H}}_1|_{x^1=0},}\\[12pt]
{\displaystyle
\partial_3\varphi=-\frac{(\widehat{\mathcal{H}}_2{H}_1-\widehat{H}_2{\mathcal{H}}_1)|_{x^1=0}}{\widehat{H}_2\widehat{\mathcal{H}}_3-\widehat{H}_3\widehat{\mathcal{H}}_2}=
a_1^2{H}_1|_{x^1=0}+a_2^2{\mathcal{H}}_1|_{x^1=0}.}
\end{array}\right.
\label{111}
\end{equation}
First we consider spatial tangential derivatives in \eqref{110}, i.e., let  $\alpha =k$ with $k=2$ or $k=3$.  By using \eqref{111} one gets
\begin{equation}
\partial_k\varphi \, \partial_kE_1|_{x^1=0}=\left. \left(a_1^k{H}_1\partial_kE_1+a_2^k{\mathcal{H}}_1\partial_kE_1\right)\right|_{x^1=0},\quad
k=2,3.
\label{112}
\end{equation}

Concerning the term $\partial_t\varphi \, \partial_tE_1|_{x^1=0}$, it can be reduced to the multiplication of spatial tangential derivatives of $U$ and $V$. Indeed, \eqref{111} and the first boundary condition in \eqref{101} imply
\begin{equation}
\partial_t\varphi =a_1^0{H}_1|_{x^1=0}+a_2^0{\mathcal{H}}_1|_{x^1=0} +v_1|_{x^1=0},
\label{113}
\end{equation}
where the constants  $a_j^0$ can be easily written down. It follows from the fourth equation in \eqref{101} and the third equation in \eqref{104} that
\begin{equation}
\partial_tE_1 = \partial_2(\mathcal{H}_3 +\varkappa E_2)-\partial_3(\mathcal{H}_2-\varkappa E_3).
\label{114}
\end{equation}
Then, \eqref{112}--\eqref{114} yield
\begin{equation}
(\partial_t\varphi \, \partial_tE_1+\partial_2\varphi \, \partial_2E_1+\partial_3\varphi \, \partial_3E_1)|_{x^1=0}=
\sum_{k=2}^{3}(M_kY,\partial_kV)|_{x^1=0},
\label{115}
\end{equation}
where $Y =(v_1,H_1,\mathcal{H}_1)$ and $M_k$ are rectangular matrices which constant coefficients depending on $(\widehat{U},\widehat{V})$ are of no interest.

To treat the integrals of the lower-order terms like $H_1\partial_kE_1|_{x^1=0}$ contained in the right-hand side of \eqref{115} we use the same standard arguments as in \cite{T05,T09,T09.cpam}. That is, we pass to the volume integral and integrate by parts:
\[
\int_{\mathbb{R}^2}H_1\partial_kE_1|_{x^1=0}\,{\rm d}x'=-\int_{\mathbb{R}^3_+}\partial_1(H_1\partial_kE_1)\,{\rm d}x  =
\int_{\mathbb{R}^3_+}\left\{ \partial_kH_1\partial_1E_1-\partial_1H_1\partial_kE_1 \right\}{\rm d}x.
\]
After that the normal ($x^1$-) derivatives of $v_1$, $H_1$ and $V$ can be expressed through tangential derivatives by using the facts that the boundary $x^1=0$ is noncharacteristic for the ``vacuum'' system \eqref{100} and $v_1$ and $H_1$ are noncharacteristic ``plasma'' unknowns.
Namely, since $\det\widehat{B}_1\neq 0$,
\begin{equation}
\partial_1{V}=(\widehat{B}_1)^{-1}(
\partial_t{V}+{B}_2\partial_2{V}+
{B}_3\partial_3{V}),
\label{117}
\end{equation}
and taking into account \eqref{35}, \eqref{35'} and the first equation in \eqref{103}, we have
\begin{align}
\left(\begin{array}{c}
\partial_1v_1 \\[3pt]
(1/\widehat{\Gamma}\,)\,\partial_1q
\end{array}
\right) &= -\frac{1}{\widehat{\Gamma}}\,\left\{J^{\sf T}\left(\widehat{A}_0\partial_t{U}+\widehat{A}_2\partial_2{U}+\widehat{A}_3\partial_3{U}+F\right)\right\}_{(12)},\label{118}
\\
\partial_1H_1 &=-\partial_2H_2-\partial_3H_3,\label{119}
\end{align}
where $\{\cdots\}_{(12)}$ denotes the first two components of the vector inside braces.
Hence, we obtain
\begin{equation}
2\hat{\mu}\sum_{\alpha =0,2,3}\int_{\mathbb{R}^2}\partial_{\alpha}\varphi \, \partial_{\alpha}E_1|_{x^1=0}{\rm d}x'  =\hat{\mu}\int_{\mathbb{R}^3_+}({\rm Q}Z,Z){\rm d}x +
\int_{\mathbb{R}^3_+}({\rm Q}_1F,Z){\rm d}x,
\label{120}
\end{equation}
where ${\rm Q}$ is a quadratic matrix of order 42 which elements can be explicitly written down if necessary,  $Z= (\partial_tU,\partial_tV,\partial_2U,\partial_2V,\partial_3U,\partial_3V)$, and ${\rm Q}_1$ is a rectangular matrix which elements are of no interest.

Now, using  \eqref{120}, we consider the sum of inequalities \eqref{110}:
\begin{equation}
\int_{\mathbb{R}^3_+}\left( (\mathfrak{A}+ \hat{\mu}{\rm Q})Z,Z)\right){\rm d}x
+\int_{\mathbb{R}^3_+}({\rm Q}_1F,Z){\rm d}x
\leq C
\left( [F]^2_{1,*T} +\sum_{\alpha =0,2,3}\int_0^tI_{\alpha}(s)\,{\rm d}s\right),
\label{121}
\end{equation}
where $\mathfrak{A}=\diag ((1/\widehat{\Gamma}\,)\widehat{A}_0,\widehat{\mathcal{B}}_0,\ldots ,(1/\widehat{\Gamma}\,)\widehat{A}_0,\widehat{\mathcal{B}}_0)$ is the block-diagonal positive definite matrix of order 42. If $\mu=|\hat{\mu}|$ is small enough, then the matrix $\mathfrak{A}+ \hat{\mu}{\rm Q}$ is also positive definite (this is enough to derive estimate \eqref{97}, see below). In principle, we can analyze the sufficient stability condition
\begin{equation}
\mathfrak{A}+ \hat{\mu}{\rm Q}>0
\label{122}
\end{equation}
if necessary. Since the matrix $\mathfrak{A}+ \hat{\mu}{\rm Q}$ is of order 42 and in RMHD the matrix $\widehat{A}_0$ is rather complicated, this is almost impossible to be done ana\-lytically, but a corresponding numerical analysis of condition \eqref{122} seems straightforward.

In this paper we restrict ourself to the assumption that $\mu$ is small enough. Using the Young inequality and the elementary inequality
\begin{equation}
\| F(t)\|^2_{L_2(\mathbb{R}^3_+)}\leq C [F]^2_{1,*,t}
\label{123}
\end{equation}
following from the trivial relation
\[
\frac{\rm d}{{\rm d}t}\,\|F(t)\|^2_{L_2(\mathbb{R}^3_+)}=2\int_{\mathbb{R}^3_+}(F,\partial_tF)\,{\rm d}x,
\]
we estimate:
\begin{equation}
\int_{\mathbb{R}^3_+}({\rm Q}_1F,Z){\rm d}x\leq C\left\{ \varepsilon\, \|Z(t)\|^2_{L_2(\mathbb{R}^3_+)} + \frac{1}{\varepsilon}\,\| F\|^2_{1,*,T} \right\},
\label{124}
\end{equation}
where $\varepsilon$ is a small positive constant. Choosing $\varepsilon$ to be small enough
and using \eqref{122} and \eqref{124},  from inequality  \eqref{121} we obtain
\begin{equation}
\|Z(t)\|^2_{L_2(\mathbb{R}^3_+)}
\leq C
\left( [F]^2_{1,*T} +\|Z\|^2_{L_2(\Omega_t)}\right).
\label{125}
\end{equation}
Here we can already use Gronwall's lemma, but we prefer to do not do this and follow arguments which can be then extended to the case of variable coefficients (see Section \ref{sec:8}).

We can easily include the $L_2$ norms of $U$, $V$, $\sigma \partial_1U$ and
$\partial_1V$ in inequality \eqref{125}. Indeed, to estimate the weighted derivatives of $U$ we do not need boundary conditions because the weight $\sigma |_{x^1=0}=0$. By applying to system \eqref{99} the operator $\sigma\partial_1$ and using standard arguments of the energy method, we obtain the inequality
\begin{equation}
\|\sigma\partial_1{U}(t)\|^2_{L_2(\mathbb{R}^3_+)} \leq C\left\{ [F]^2_{1,*,T}+[{U}]^2_{1,*,t} \right\}\label{126}
\end{equation}
(actually for the present case of constant coefficients the norm $[{U}]^2_{1,*,t}$ in \eqref{126} can be even replaced by $\|\sigma\partial_1{U}\|^2_{L_2(\Omega_t)}$). It follows from \eqref{117} that
\begin{equation}
\|\partial_1V(t)\|^2_{L_2(\mathbb{R}^3_+)} \leq C\|V\|^2_{H^1(\Omega_t)}.\label{127}
\end{equation}
Summing up inequalities \eqref{125}--\eqref{127} and the elementary inequalities for $U$ and $V$ like  \eqref{123},
\begin{equation}
\| U(t)\|^2_{L_2(\mathbb{R}^3_+)}\leq C [U]^2_{1,*,t},\qquad
\| V(t)\|^2_{L_2(\mathbb{R}^3_+)}\leq C \|V\|^2_{H^1(\Omega_t)},
\label{128}
\end{equation}
we get
\begin{equation}
\nt U(t)\nt ^2_{1,*}+\nt V(t)\nt ^2_{H^1(\mathbb{R}^3_+)}
\leq C
\left( [F]^2_{1,*T} +[{U}]^2_{1,*,t}+ \|V\|^2_{H^1(\Omega_t)}\right),
\label{129}
\end{equation}
with
\[
\nt V (t)\nt^2_{H^1(\mathbb{R}^3_+)}:=\|V(t)\|^2_{H^1(\mathbb{R}^3_+)}+
\|\partial_tV(t)\|^2_{L_2(\mathbb{R}^3_+)}.
\]

It remains to estimate the interface perturbation $\varphi$. By using \eqref{118} and \eqref{119} we first estimate the noncharacteristic part $U_n=(q,v_1,H_1)$ of the unknown $U$:
\begin{equation}
\|\partial_1U_n (t)\|^2_{L_2(\mathbb{R}^3_+)}\leq C(K)\left\{[F]^2_{1,*,T}+\nt {U}(t)\nt^2_{1,*}\right\}.
\label{130}
\end{equation}
Then thanks to the trace theorems, \eqref{127} and \eqref{130} imply
\begin{equation}
\|U_n|_{x^1=0} (t)\|^2_{L_2(\mathbb{R}^2)} \leq C \left\{[F]^2_{1,*,T}+\nt \dot{U}(t)\nt^2_{1,*}\right\},
\label{131}
\end{equation}
\begin{equation}
\|V|_{x^1=0} (t)\|^2_{L_2(\mathbb{R}^2)}  \leq C\|V\|^2_{H^1(\Omega_t)}.
\label{132}
\end{equation}
The multiplication of the first boundary condition in \eqref{101} by $2\varphi$, the integration over the domain $\mathbb{R}^2$, and the usage of \eqref{111}, \eqref{113}, estimates \eqref{131} and \eqref{132} result in
\begin{equation}
\nt\varphi (t)\nt^2_{H^1(\mathbb{R}^2)}\leq C\Bigl\{ [F]^2_{1,*,T}
+  [{U}]^2_{1,*,t} +\|V\|^2_{H^1(\Omega_t)}+\|\varphi\|^2_{L_2(\partial\Omega_t)}\Bigr\}.
\label{133}
\end{equation}
Summing up inequalities \eqref{129} and \eqref{133} we obtain
\begin{multline}
\nt U(t)\nt ^2_{1,*}+\nt V(t)\nt ^2_{H^1(\mathbb{R}^3_+)}+
\nt\varphi (t)\nt^2_{H^1(\mathbb{R}^2)} \\ \leq C\Bigl\{ [F]^2_{1,*,T}
+  [{U}]^2_{1,*,t} +\|V\|^2_{H^1(\Omega_t)}+\|\varphi\|^2_{H^1(\partial\Omega_t)}\Bigr\}.
\label{134}
\end{multline}
Applying Gronwall's lemma, from \eqref{134} we derive the desired a priori estimate \eqref{97}.

\begin{remark}{\rm
In fact, for problem \eqref{99}--\eqref{102} we have constructed a dissipative 1-symmetrizer \cite{T06}. The index 1 means that we prolong our hyperbolic system up to first-order derivatives and manage to construct a secondary symmetrization of the prolonged system with dissipative boundary conditions (see \cite{T06}). The divergence constrains \eqref{103} play a crucial role in this construction. Note that in \cite{T05} we constructed a dissipative 0-symmetrizer and proved an $L_2$ estimate for the constant coefficients linearized problem for current-vortex sheets and then derived an estimate in $H^1_*$ for the variable coefficients problem. Here, for the plasma-vacuum interface problem we have to prolong the hyperbolic system up to first-order derivatives already on the level of constant coefficients.}
\label{r3}
\end{remark}

\section{Normal modes analysis for particular cases}
\label{sec:7}

There appears the following natural question. Can problem \eqref{99}--\eqref{102} be really ill-posed or just our approach in Section \ref{sec:6} does not allow us to derive an energy a priori estimate for the whole range of admissible parameters of this problem? Below we show that this problem can be indeed ill-posed. In other words, unlike planar fluid-vacuum boundaries \cite{Lind,T09.cpam} and planar nonrelativistic plasma-vacuum interfaces \cite{T10},
a planar relativistic plasma-vacuum interface can be violently unstable, i.e., {\it Kelvin-Helmholz instability} can occur. Mathematically this means that there is a range of admissible parameters of problem \eqref{99}--\eqref{102} for which the Kreiss-Lopatinski condition \cite{Kreiss,Maj83b} is violated.

Unfortunately, due to principal technical difficulties a complete Fourier-Laplace analysis of problem \eqref{99}--\eqref{101} seems impossible (even numerically). Therefore, now our goal is just to find particular examples of the ill-posedness of problem \eqref{99}--\eqref{101}.
We restrict ourself to a case when the structure of the matrices $A_{\alpha}$ in \eqref{8} is not so complicated and we can perform analytical calculations. This is the case
when the 3-vectors of velocity and magnetic field are collinear and have only one nonzero component. Such a case was already considered in \cite{T99} for relativistic MHD shock waves.

The test of the Kreiss-Lopatinski condition is equivalent to the usual normal modes analysis, i.e., the construction of an Hadamard-type ill-posedness example. We seek exponential solutions to problem \eqref{99}--\eqref{101} with $F=0$ and special initial data:
\begin{align}
& U = \bar{U}\exp \left\{ \tau t +\xi_{\rm p}x^1+ i(\gamma ', x') \right\},\label{135}\\
& V = \bar{V}\exp \left\{ \tau t +\xi_{\rm v}x^1+ i(\gamma ', x') \right\},\label{136}\\
& \varphi = \bar{\varphi}\exp \left\{ \tau t + i(\gamma ', x') \right\},\label{137}
\end{align}
where $\bar{U}$ and $\bar{V}$ are complex-valued constant vectors, $\tau$, $\xi_{\rm p}$, $\xi_{\rm v}$ and $\bar{\varphi}$ are complex constants, and $\gamma '=(\gamma_2,\gamma_3)$ with real constants $\gamma_{2,3}$. The existence of exponentially growing solutions in
form \eqref{135}--\eqref{137}, with
\begin{equation}
\Re\,\tau>0,\quad \Re\,\xi_{\rm p}<0,\quad
\Re\,\xi_{\rm v}<0,\label{138}
\end{equation}
implies the ill-posedness of problem \eqref{99}--\eqref{101} because the consequence of solutions
\[
(U (nt,n{x} ), V( nt, nx), \varphi (nt,n{x}' ))\exp
(-\sqrt{n}\,),\quad n=1,2,3,\ldots\,,
\]
with $U$, $V$ and $\varphi$ defined in \eqref{135}--\eqref{137}, is the Hadamard-type ill-posedness example. Since the last equation in \eqref{99} for the entropy perturbation $S$ plays no role in the construction of an ill-posedness example, we will suppose that $U=(p,u,H)$.

Theoretically, we can construct a 1D ill-posedness example (with $\gamma' =0$) on a codimension-1 set in the parameter domain. But, this is not the case for problem \eqref{99}--\eqref{101}. Indeed, in 1D we have $E_1=0$ and the energy inequality \eqref{109}, which can be obtained even for $\varkappa \geq 0$, implies well-posedness.

Thus, we may assume that $\gamma' \neq 0$ and even, without loss of generality, $|\gamma'|=1$.
Moreover, in view of the big complexity of the RMHD system, we will construct a 2D ill-posedness example with $\gamma_3=0$, i.e., $\gamma_2 =1$. Let also consider the particular case $\hat{v}_1=\hat{v}_2=0$ and $\widehat{H}_2=0$ for the relativistic magnetoacoustic system \eqref{109} (recall that $\widehat{H}_1=0$, see \eqref{98}). Since $\hat{v}_1=\varkappa$, we have $\varkappa =0$. However, by continuity the below ill-posedness examples can be extended to the case $|\varkappa | \ll 1$. For the particular case $\hat{v}_1=\hat{v}_2=\widehat{H}_1=\widehat{H}_2=0$ the matrices $\widehat{A}_{\alpha}$  (without the last rows and columns) take the following ``visible'' form:
\[
\widehat{A}_0=\left( \begin{array}{ccccccc}
\widehat{\Gamma}/(\hat{\rho}\hat{a}^2)&0&0&\hat{v}_3&0&0&0\\
0&\hat{\rho}\hat{h}\widehat{\Gamma}+\widehat{H}_3^2/\widehat{\Gamma}&0&0&0&0&0\\
0&0&\hat{\rho}\hat{h}\widehat{\Gamma}+\widehat{H}_3^2/\widehat{\Gamma}&0&0&0&0\\
\hat{v}_3&0&0&\hat{\rho}\hat{h}/\widehat{\Gamma} &0&0&0\\
0&0&0&0&1/\widehat{\Gamma}&0&0\\
0&0&0&0&0&1/\widehat{\Gamma}&0\\
0&0&0&0&0&0&\widehat{\Gamma} \end{array}\right),
\]
\[
\widehat{A}_1=\left( \begin{array}{ccccccc}
0&1&0&0&0&0&0\\
1&0&0&0&0&0&\widehat{H}_3\\
0&0&0&0&0&0&0\\
0&0&0&0 &0&0&0\\
0&0&0&0&0&0&0\\
0&0&0&0&0&0&0\\
0&\widehat{H}_3&0&0&0&0&0 \end{array}\right),\quad
\widehat{A}_2=\left( \begin{array}{ccccccc}
0&0&1&0&0&0&0\\
0&0&0&0&0&0&0\\
1&0&0&0&0&0&\widehat{H}_3\\
0&0&0&0 &0&0&0\\
0&0&0&0&0&0&0\\
0&0&0&0&0&0&0\\
0&0&\widehat{H}_3&0&0&0&0 \end{array}\right)
\]
(see also \cite{T99} for the analogous case $\hat{v}_2=\hat{v}_3=\widehat{H}_2=\widehat{H}_3=0$). Since $\gamma_3=0$, the matrix $\widehat{A}_3$ is of no interest and we do not write down it here.

Substituting \eqref{135} and \eqref{136} into \eqref{99} (with $F=0$) and \eqref{100} respectively, we get the dispersion relations
\[
\det (\tau \widehat{A}_0 +\xi_{\rm p}\widehat{A}_1+i\widehat{A}_2)=0,\quad
\det (\tau I -\xi_{\rm v}B_1+iB_2)=0
\]
from which, omitting technical calculations, we find the unique roots
\begin{equation}
\xi_{\rm p}=-\sqrt{1+r\tau^2}, \qquad \xi_{\rm v}=-\sqrt{1+\tau^2},
\label{139}
\end{equation}
with property \eqref{138}, where
\[
r=\frac{w^2(\widehat{\Gamma}^2+m^2\hat{c}_{s}^2)}{\hat{c}_{s}^2(1+m^2w^2)},\quad
w^2=1-\hat{v}_3^2\hat{c}_{s}^2>0,\quad m=\frac{\widehat{H}_3}{\hat{a}\sqrt{\hat{\rho}}}.
\]

Consider first the particular case $\widehat{\mathcal{H}}_2=0$. For this case our basic assumption \eqref{52a} is violated. Substituting \eqref{135}--\eqref{137} into the boundary conditions \eqref{101}, we obtain
\[
\tau \bar{\varphi}=\bar{v}_1,\quad
\bar{q}=\widehat{\mathcal{H}}_3\bar{\mathcal{H}}_3-\widehat{E}_1\bar{E}_1,\quad
\bar{E}_{2}=\tau\widehat{\mathcal{H}}_3\bar{\varphi}
-i\widehat{E}_1\bar{\varphi} ,
\]
and excluding the constant $\bar{\varphi}$ we come to the relations
\begin{equation}
\bar{q}+\widehat{E}_1\bar{E}_1-\widehat{\mathcal{H}}_3\bar{\mathcal{H}}_3=0,\quad
(i\widehat{E}_1-\tau\widehat{\mathcal{H}}_3)\bar{v}_1+\tau\bar{E}_{2}=0.\label{140}
\end{equation}
From \eqref{135} and \eqref{136} we have
\begin{equation}
\xi_{\rm p}\bar{q}+\tau
\Bigl(\hat{\rho}\hat{h}\widehat{\Gamma}+\frac{\widehat{H}_3^2}{\widehat{\Gamma}}\,\Bigr)\bar{v}_1=0,\quad
\tau \bar{E}_1=i\bar{\mathcal{H}}_3,\quad
i\bar{E}_2=\xi_{\rm v} \bar{E}_1  .\label{141}
\end{equation}
Then, \eqref{140} and \eqref{141} imply
\begin{equation}
\left(
\begin{array}{ccc}
1  & 0 & \widehat{E}_1+i\tau \widehat{\mathcal{H}}_3 \\[3pt]
\xi_{\rm p} & \;\,\tau(\hat{\rho}\hat{h}\widehat{\Gamma}+\widehat{H}_3^2/\widehat{\Gamma})\; & 0 \\[3pt]
0 & i\widehat{E}_1-\tau\widehat{\mathcal{H}}_3 & -i\tau\xi_{\rm v}
\end{array}
\right)\left(
\begin{array}{c} \bar{q} \\ \bar{v}_1 \\ \bar{E}_1\end{array}\right) =0.\label{142}
\end{equation}

Since we are interesting in solutions with $(\bar{U},\bar{V})\neq 0$, the determinant of the matrix in \eqref{142} should be equal to zero. Taking into account \eqref{139}, this gives us the following final equation for $\tau$:
\begin{equation}
\tau^2
\sqrt{1+\tau^2} = \ell ( \widehat{E}_1+i\widehat{\mathcal{H}}_3\tau )^2\sqrt{1+r\tau^2},
\label{143}
\end{equation}
where
\[
\ell =\frac{\widehat{\Gamma}}{\hat{\rho}\hat{h}\widehat{\Gamma}^2+\widehat{H}_3^2} >0,\qquad
r>1.
\]
If $\widehat{E}_1=\widehat{\mathcal{H}}_3=0$, then $\mu =0$ and equation \eqref{143} has only roots with $\Re\,\tau =0$. That is, at least for the 2D problem (recall that $\gamma_3=0$) the Kreiss-Lopatinski is satisfied in a weak sense. Note that for the case $\mu =0$ we can derive an $L_2$ a priori estimate for problem \eqref{99}--\eqref{102} (in 3D) even if assumption \eqref{52a} is violated. This follows from \eqref{110}.

If $\widehat{\mathcal{H}}_3=0$ and  $\widehat{E}_1\neq 0$, then graphically it is easy to see that equation \eqref{143} has a root $\tau^2>0$, i.e., we have an ``unstable'' root $\tau >0$. Clearly, there is an ``unstable'' root $\tau$ with $\Re\,\tau >0$ at least if $|\widehat{\mathcal{H}}_3|$ is small enough. It seems that one can prove that \eqref{143} has a root with property \eqref{138} even for arbitrary values of $\widehat{\mathcal{H}}_3$. We do not study this technical problem here. We only note that for the case $r\rightarrow 1$ which correspond to the ultrarelativistic limit $\hat{c}_s\rightarrow 1$ equation \eqref{143} has a root with $\Re\,\tau >0$ for arbitrary $\widehat{\mathcal{H}}_3$ and $\widehat{E}_1\neq 0$.
Thus, we see that if assumption \eqref{52a} is violated, then the linearized problem can be ill-posed.

We now consider a particular case when assumption \eqref{52a} holds. Let $\widehat{\mathcal{H}}_3=0$ and $\widehat{H}_3\widehat{\mathcal{H}}_2\neq 0$. Substituting \eqref{135}--\eqref{137} into the boundary conditions \eqref{101}, after the exclusion of
constant $\bar{\varphi}$ we obtain
\begin{equation}
\bar{q}+\widehat{E}_1\bar{E}_1-\widehat{\mathcal{H}}_2\bar{\mathcal{H}}_2=0,\quad
i\widehat{E}_1\bar{v}_1+\tau\bar{E}_{2}=0,\quad \widehat{\mathcal{H}}_2\bar{v}_1+\bar{E}_{3}=0.\label{145}
\end{equation}
It follows from \eqref{99} and \eqref{100} that
\begin{equation}
\xi_{\rm p}\bar{q}+\tau
\Bigl(\hat{\rho}\hat{h}\widehat{\Gamma}+\frac{\widehat{H}_3^2}{\widehat{\Gamma}}\,\Bigr)\bar{v}_1=0,\quad
\xi_{\rm v} \bar{E}_3=\tau\bar{\mathcal{H}}_2,\quad
i\bar{E}_2=\xi_{\rm v} \bar{E}_1  .\label{146}
\end{equation}
The algebraic system \eqref{145}, \eqref{146} has a nonzero solution $(\bar{q},\bar{v}_1,\bar{E})$ if
\begin{equation}
\tau^2
\sqrt{1+\tau^2} = \ell \bigl( \widehat{E}_1^2-\widehat{\mathcal{H}}_2^2 (1+\tau^2)\bigr)\sqrt{1+r\tau^2}
\label{147}
\end{equation}
(we omit intermediate simple calculations) or $G(y)=0$, where $y=\tau^2$ and
\[
G(y)=y\sqrt{1+y} - \ell \bigl( \widehat{E}_1^2-\widehat{\mathcal{H}}_2^2 (1+y)\bigr)\sqrt{1+ry}.
\]

If $\widehat{E}_1^2>\widehat{\mathcal{H}}_2^2$, then $G(0)<0$ and $G(y^*)>0$, where $y^*=(\widehat{E}_1^2/\widehat{\mathcal{H}}_2^2)-1 >0$ (recall that we consider the case $\widehat{\mathcal{H}}_2\neq 0$). Hence, if
\begin{equation}
\widehat{E}_1^2>\widehat{\mathcal{H}}_2^2,
\label{148}
\end{equation}
then equation \eqref{147} has an ``unstable'' root $\tau >0$, i.e., problem \eqref{99}--\eqref{102} is ill-posed. If
\begin{equation}
\widehat{E}_1^2\leq\widehat{\mathcal{H}}_2^2,
\label{149}
\end{equation}
at least for $r\rightarrow 1$ it is easy to see that equation \eqref{147} has only roots with $\Re\,\tau =0$. That is, again the Kreiss-Lopatinski can be satisfied only in a weak sense.

For the particular case under consideration, $\mu =|\widehat{E}_1-\hat{v}_3\widehat{\mathcal{H}}_2|$. In Section \ref{sec:6} we could not unfortunately concretize how small should be $\mu$ to guarantee the fulfillment of \eqref{122}. Clearly, if $\widehat{E}_1$ is close to $\hat{v}_3\widehat{\mathcal{H}}_2$, then \eqref{122} holds.
If, for example, $\hat{v}_3=1/2$, we see that the condition $\widehat{E}_1\approx \hat{v}_3\widehat{\mathcal{H}}_2$ is more restrictive than \eqref{149}. In general, if $\widehat{E}_1$ is close to $\hat{v}_3\widehat{\mathcal{H}}_2$, the instability condition \eqref{148} is violated.

\section{Variable coefficients analysis}
\label{sec:8}

The principal difference of the derivation of estimate \eqref{97} for the variable coefficients case in comparison with the constant coefficients one is the appearance of additional ``not dangerous'' lower-order terms and the fact that the usage of spaces $H^m_*$ is a principal point.\footnote{For the constant coefficients problem we could estimate the ``plasma'' unknown $U$ in the norm $\|U\|_{H^1_{\rm tan}}=\|U\|_{L_2}+\|\partial_tU\|_{L_2}+\|\partial_2U\|_{L_2}+\|\partial_3U\|_{L_2}.$}
We now give short comments about the adaptation of arguments from Section \ref{sec:6}
to the variable coefficients problem \eqref{89}--\eqref{92}.

First of all, for problem \eqref{89}--\eqref{92} we can easily get estimates \eqref{126}--\eqref{128}, \eqref{130}--\eqref{132}, where now $U_n=(q,v_N,H_N)$. For example, the usage of the norms of $H^1_*$ in the right-hand side of \eqref{130} is really necessary because the boundary matrix $J^{\sf T}\widehat{A}_1J=\mathcal{E}_{12} +\mathcal{N}$ (see \eqref{35}, \eqref{35'}) and we can guarantee that the matrix $\mathcal{N}$ is zero only on the boundary: $\mathcal{N}|_{x^1}=0$. Concerning the preparatory estimate \eqref{133}, by virtue \eqref{95}, \eqref{96} and assumption \eqref{52}, we can also obtain it because now in \eqref{111}, \eqref{113} we just have additional zero-order terms in $\varphi$ and should replace $H_1$ by $H_N$, $\mathcal{H}_1$ by $\mathcal{H}_N$, and $v_1$ by $v_N$.

Using \eqref{94}, we get system \eqref{81} for $V$ and our choice is again \eqref{106} (but now $\hat{v}_j$ are not constants). Then, we deduce the following energy inequality for $\partial_{\a}U$ and $\partial_{\a}V$ with ${\a}=0,2,3$:
\begin{equation}
I_{\a}(t)+ 2\int_{\partial\Omega_t}\mathcal{Q}_{\a}\,{\rm d}x'{\rm d}s \\ \leq C
\left\{ [F]^2_{1,*,T} +[{U}]^2_{1,*,t} +\|V\|^2_{H^1(\Omega_t)}\right\},
\label{150}
\end{equation}
where $I_{\a}$ were determined just after \eqref{110} and
\begin{multline*}
\mathcal{Q}_{\a}=\frac{1}{2}(\widehat{\mathcal{B}}_1\partial_{\a}V,\partial_{\a}V)|_{x^1=0}-\frac{1}{2\widehat{\Gamma}}(\widehat{A}_1\partial_{\a}U,\partial_{\a}U)|_{x^1=0} \\=
\frac{1}{2}(\widehat{\mathcal{B}}_1\partial_{\a}V,\partial_{\a}V)|_{x^1=0}-\partial_{\a}{q}\,\partial_{\a}{v}_N|_{x^1=0}.
\end{multline*}

We first calculate the quadratic form $\mathcal{Q}$ determined in Section \ref{sec:6} (see \eqref{107}):
\[
\mathcal{Q}= \left.\left\{\hat{v}_2\mathcal{H}_1\mathcal{H}_2
+\hat{v}_3\mathcal{H}_1\mathcal{H}_3+(\hat{v}_2\partial_2\hat{\varphi}+\hat{v}_3\partial_3\hat{\varphi})\mathcal{H}_1^2+\ldots -qv_N\right\}\right|_{x^1=0}.
\]
This is a point when we have to perform {\it big} technical calculations. Omitting these rather long calculations which are really straightforward, we obtain
\[
\mathcal{Q}=\hat{\mu}\left.\left\{ {E}_N\partial_t\varphi+({\mathcal{H}}_{\tau_2}+\varkappa E_3)\partial_3\varphi -({\mathcal{H}}_{\tau_3}-\varkappa E_2)\partial_2\varphi\right\}\right|_{x^1=0}+\mbox{l.o.t.},
\]
where l.o.t. is the sum of lower-order terms like $[ \partial_1\hat{q}]\,v_N|_{x^1=0}\,\varphi$.
While obtaining the above formula for $\mathcal{Q}$ we used assumptions  \eqref{40}, the boundary conditions \eqref{91} and \eqref{96}. Taking into account the equation
\[
\partial_t{\mathfrak{e}}-\nabla^-\times {\mathfrak{H}}+\varkappa\,\partial_1{\mathfrak{e}}+ B(\hat{\varphi} )\mathfrak{e}=0
\]
following from \eqref{90} and applying the second equation in \eqref{94}, we rewrite $\mathcal{Q}$ as follows:
\[
\mathcal{Q}=\partial_t\left(\hat{\mu}\varphi E_N|_{x^1=0} \right)+\partial_2\left(\hat{\mu}\varphi(\varkappa E_2-\mathcal{H}_{\tau_3})|_{x^1=0} \right)  + \partial_3\left(\hat{\mu}\varphi (\mathcal{H}_{\tau_2}+\varkappa E_3)|_{x^1=0}\right)+\mbox{l.o.t.}
\]

Then we have
\begin{equation}
\int_{\partial\Omega_t}\mathcal{Q}_{\a}\,{\rm d}x'\,{\rm d}s=
\int_{\mathbb{R}^2}\hat{\mu}\partial_{\a}\varphi \partial_{\a}E_N|_{x^1=0}\,{\rm d}x'+\int_{\partial\Omega_t}(\mbox{l.o.t.}) \,{\rm d}x'{\rm d}s,
\label{150'}
\end{equation}
where l.o.t. is the sum of lower-order terms like
\[
\hat{c}\,\partial_{\a}v_N|_{x^1=0}\partial_{\a}\varphi ,\quad
\hat{c}\,\partial_{\a}q|_{x^1=0}\partial_{\a}\varphi ,\quad
\hat{c}\,\partial_{\a}\mathcal{H}_j|_{x^1=0}\partial_{\a}\varphi ,\quad
\hat{c}\,\partial_{\a}E_j|_{x^1=0}\partial_{\a}\varphi ,
\]
etc., here and below $\hat{c}$ is the common notation of a coefficient depending on the basic state \eqref{37}. By using the variable coefficients counterparts of \eqref{111} and \eqref{113}
(where we have additional zero-order terms in $\hat{\varphi}$; $H_1$ is replaced by $H_N$, etc.), we reduce the lower-order terms above to those like
\[
\hat{c}\,H_N\partial_{\a}v_N|_{x^1=0},\quad
\hat{c}\,\mathcal{H}_N\partial_{\a}\mathcal{H}_j|_{x^1=0},\quad \hat{c}\,H_N\partial_{\a}E_j|_{x^1=0},\quad \hat{c}\,\mathcal{H}_N\partial_{\a}q|_{x^1=0},
\]
etc., or even ``better'' terms like $\hat{c}\varphi \partial_{\a}E_j|_{x^1=0}$. By integration by parts such ``better'' terms can be reduced to the above ones and terms like $\hat{c}\varphi E_j|_{x^1=0}$. Actually, if $\a =0$, then the integration by parts leads to the appearance of terms in the form $\partial_t(\cdots )$ which do not disappear after the integration over the domain $\partial\Omega_t$. In this case we use the Young inequality (see explanations below).

We treat the terms like $\hat{c}\,H_N\partial_{\a}v_N|_{x^1=0}$  by passing to the volume integral and integrating by parts (the only difference from the analogous manipulation in Section \ref{sec:6} is the appearance of variable coefficients):
\begin{multline*}
\int_{\partial\Omega_t}\hat{c}\,{H}_N\partial_{\a}{v}_N|_{x^1=0}\,{\rm d}x'\,{\rm d}s
=-\int_{\Omega_t}\partial_1\bigl(\tilde{c}{H}_N\partial_{\a}{v}_N\bigr){\rm d}x\,{\rm d}s \\
=\int_{\Omega_t}\Bigl\{\tilde{c}\partial_{\a}{H}_N\partial_1{v}_N +(\partial_{\a}\tilde{c}){H}_N\partial_1{v}_N -\tilde{c}\partial_1{H}_N\partial_{\a}{v}_N-(\partial_1\tilde{c}){H}_N\partial_1{v}_N\Bigr\}{\rm d}x\,{\rm d}s \\
-\int_{\Omega_t}\partial_{\a}\bigl(\tilde{c}{H}_N\partial_1{v}_N\bigr){\rm d}x\,{\rm d}s,
\end{multline*}
where $\tilde{c}|_{x^1=0}=\hat{c}$. If $\a =2$ or $\a =3$ the last integral above is equal to zero. But for $\a =0$ we have:
\[
-\int_{\Omega_t}\partial_s\bigl(\tilde{c}{H}_N\partial_1{v}_N\bigr){\rm d}x\,{\rm d}s=
-\int_{\mathbb{R}^3_+}\tilde{c}{H}_N\partial_1{v}_N{\rm d}x\,{\rm d}s.
\]
Using the Young inequality, \eqref{128} and \eqref{130} (with $U_n=(q,v_N,H_N)$), we estimate the last integral as follows:
\[
-\int_{\mathbb{R}^3_+}\tilde{c}{H}_N\partial_{1}{v}_N\,{\rm d}x\leq C\left\{
[F]^2_{1,*,T}+\tilde{\varepsilon}\,\nt {U}(t)\nt^2_{1,*} +\frac{1}{\tilde{\varepsilon}}\,
[{U}]^2_{1,*,t}\right\},
\]
where $\tilde{\varepsilon} $ is a small positive constant. Concerning the lower-order terms $\hat{c}\varphi E_j|_{x^1=0}$, $\hat{c}\varphi H_N|_{x^1=0}$, etc., we treat them by using the trace theorems in $H^1$ and $H^1_*$.

If $\a =2$ or $\a =3$, we apply to the boundary integral in \eqref{150'} the same arguments as for the last integral there and take into account \eqref{126}--\eqref{128}, \eqref{130}--\eqref{133}. If $\a =0$, we first use the variable coefficients counterparts of \eqref{113}, \eqref{114} and then apply the same arguments as for the cases $\a =2$ and $\a =3$.
In particular, for variable coefficients in \eqref{114} we have $E_N$ instead of $E_1$, $\mathcal{H}_{\tau_2}$ instead of $\mathcal{H}_2$, and $\mathcal{H}_{\tau_3}$ instead of $\mathcal{H}_3$.

Thus, choosing $\tilde{\varepsilon} $ and the constant  $\varepsilon$ from the variable coefficients counterpart of \eqref{124}  to be small enough, we finally get inequality \eqref{134}, provided that the constant $\mu^*$ from Theorem \ref{t1} is also small enough.\footnote{Actually, for variable coefficients the less restrictive sufficient stability condition is the same as in \eqref{122}. The only difference is that coefficients of $\mathfrak{A}$ and $\mathcal{Q}$ are now functions depending on $\widehat{U}$ and $\widehat{V}$.} Then, applying Gronwall's lemma, we derive the desired a priori estimate \eqref{97}. We did not discuss assumption \eqref{38}, but we can carefully check that this assumption on the coefficients is indeed optimal. In this connection, we refer to \cite{T05} where the analogous assumption appeared for the linearized problem for current-vortex sheets.

\section{Concluding remarks}
\label{sec:9}

We have derived a basic a priori estimate for the linearized plasma-vacuum interface problem in RMHD under a sufficient stability condition. This condition precludes the violent instability of planar relativistic plasma-vacuum interfaces. We have also proved that, unlike the situation for fluid-vacuum free boundaries \cite{Lind,T09.cpam} and nonrelativistic plasma-vacuum interfaces \cite{T10}, the constant coefficients linearized plasma-vacuum problem in RMHD can be ill-posed due to a large enough vacuum electric field, i.e., Kelvin-Helmholz instability can really occur.

In the basic a priori estimate \eqref{54} obtained for the  variable coefficients linearized problem we have a loss of derivatives from the source terms to the solution. Therefore, we plan to prove the existence of solutions to the original nonlinear problem by a suitable Nash-Moser-type iteration scheme. We suppose that the scheme of the proof will be really similar to that in \cite{CS,T09,T09.cpam}. The important point in this direction is the derivation of a so-called tame estimate for solutions $(\dot{U},\dot{V},\varphi )$ of the linearized problem in $H^s_*(\Omega_T)\times H^s(\Omega_T)\times H^s(\partial\Omega_T)$ for $s$ large enough. 

In this paper we restricted ourselves to the case of special relativity. At first sight, to cover general relativity we could try to follow a scheme proposed in \cite{T09.cpam} for the fluid-vacuum problem.  Namely, in \cite{T09.cpam} we  proposed to resolve the matter equations by Nash-Moser iterations whereas at each Nash-Moser iteration step we find the metric from the Einstein equations by Picard iterations. It is quite possible that such a scheme of the proof of the existence of solutions to the nonlinear problem could work (however, we are not yet absolutely sure in success even for the general relativistic fluid-vacuum problem), but it is still unclear how to prove the uniqueness of a solution.

The point is that to prove uniqueness we have to derive an priori estimate for the linearized problem with all the lower-order terms. It means that we have to pass to the Alinhac's ``good unknown'' even in the linearized Einstein equations written in the form of a symmetric hyperbolic system \cite{Rendall}. That is the linear boundary conditions $[\partial_{\gamma}{g}_{\alpha\beta}]=0$ for the metric $g$ (see \cite{T09.cpam}) after the passage to the ``good unknown'' $\dot{g}_{\alpha\beta}$ become
\[
[\partial_{\gamma}\dot{g}_{\alpha\beta}]=-\varphi [\partial_{1}\partial_{\gamma}\hat{g}_{\alpha\beta}],
\]
where $\hat{g}_{\alpha\beta}$ is an ``unperturbed'' metric and $[\cdot ]$ denotes a jump at $x^1=0$. The right-hand sides in the above boundary conditions can be nonzero and the fact that for the fluid-vacuum problem the symbol associated to the free surface is not elliptic prevents obtaining the a priori estimate.

It is astonishing that for the general relativistic plasma-vacuum interface problem we do not have such a difficulty at all because there is a natural stabilizing physical mechanism. This mechanism is the magnetic field (!). Indeed, under assumption \eqref{52} the interface symbol is elliptic and, therefore, the lower-order terms appearing in the boundary conditions for the metric $g$ after the passage to the ``good unknown'' are not dangerous anymore. Thus, in the future after the proof of a nonlinear existence theorem for problem
\eqref{8}, \eqref{10}, \eqref{13}, \eqref{14}, \eqref{16}, \eqref{17} we hope to extend this result to the general relativistic version of this problem.

\vspace*{12pt}
\noindent
{\bf Acknowledgements}
\vspace*{9pt}

This work was supported by RFBR (Russian Foundation for Basic Research) grant No. 10-01-00320-a.

\end{document}